\documentclass[a4paper,12pt]{amsart}
\usepackage{lineno}
\usepackage[utf8]{inputenc}
\usepackage[T1]{fontenc}
\usepackage[UKenglish]{babel}
\usepackage[margin=18mm]{geometry}
\usepackage{color}%cyan,red;magenta,green;yellow,blue

\usepackage{graphicx}

\usepackage{amsmath,amssymb,amsfonts,amsthm,amscd,amsxtra}
\usepackage{mathrsfs,eucal,dsfont}
\usepackage{verbatim,enumitem}

\usepackage{hyperref,url}

\newcommand{\R}{\mathds R}

\def\bb{\alpha}
\def\dd{\delta}
\def\d{{\rm d}}
\def\<{\langle}
\def\>{\rangle}

 \def\tt{\tilde}
 \def\ff{\frac}
 \def\sq{\sqrt}
\def\bb{\beta}
\def \qq{\quad}
\def\R{\mathbb R}  \def\ff{\frac} \def\sq{\sqrt} 
 \def\kk{\kappa} 
\def\dd{\delta}  \def\vv{\varepsilon} 
\def\<{\langle} \def\>{\rangle} \def\bb{\Gamma} \def\bb{\gamma}
  \def\nn{\nabla}  
\def\d{\text{\rm{d}}} \def\bb{\beta} \def\aa{\alpha} 
  \def\si{\sigma} 
 \def\beq{\begin{equation}}  \def\F{\mathscr F}
 
\def\e{\text{\rm{e}}}  \def\OO{\Omega}  \def\oo{\omega}
 \def\tt{\tilde} 
 \def\P{\mathbb P} 
\def\C{\mathscr C}           
   
  \def\ll{\lambda}
 
\def\E{\mathbb E}

\def\to{\rightarrow}
\def\8{\infty}\def\3{\triangle}
\def\W{\mathbb{W}}\def\1{\lesssim}

\renewcommand{\bar}{\overline}
\renewcommand{\hat}{\widehat}
\renewcommand{\tilde}{\widetilde}

\newtheorem{theorem}{Theorem}[section]
\newtheorem{lemma}[theorem]{Lemma}

\newtheorem{corollary}[theorem]{Corollary}

\theoremstyle{definition}

\newtheorem{remark}[theorem]{Remark}

\numberwithin{equation}{section}
\begin{document}
\allowdisplaybreaks
%\linenumbers

\title[Uniform-in-time  estimates  for mean-field type SDEs] {Uniform-in-time  estimates  for mean-field type SDEs and applications}

\author{
Jianhai Bao\qquad
Jiaqing Hao
}
\date{}
\thanks{\emph{J.\ Bao:} Center for Applied Mathematics, Tianjin University, 300072  Tianjin, P.R. China. \url{jianhaibao@tju.edu.cn}}

\thanks{\emph{J.\ Hao:}
Center for Applied Mathematics, Tianjin University, 300072  Tianjin, P.R. China. \url{hjq_0227@tju.edu.cn}}

\maketitle

\begin{abstract}
Via constructing an asymptotic coupling by reflection, in this paper we establish uniform-in-time estimates on probability distances for mean-field type SDEs, where
the drift terms under consideration are dissipative merely in the long distance. As applications, we (i) explore the long time probability distance estimate between an SDE and its delay version; (ii) investigate the issue on uniform-in-time propagation of chaos for McKean-Vlasov SDEs, where the drifts might be singular with respect to the spatial variables;
  (iii) tackle  the discretization error bounds in an infinite-time horizon for stochastic algorithms (e.g. backward/tamed/adaptive Euler-Maruyama schemes as three typical candidates) associated with McKean-Vlasov SDEs.

\medskip

\noindent\textbf{Keywords:} Asymptotic coupling by reflection, mean-field type SDE, propagation of chaos, stochastic algorithm, uniform-in-time estimate.

\smallskip

\noindent \textbf{MSC 2020:} 60H10, 60F05, 60F15.
\end{abstract}

\section{Introduction, main result and applications}
\subsection{Introduction}\label{subsec1}\label{sub1}In this paper, we work on the  mean-field type   SDEs: for   $i\in\mathbb S_N:=\{1,\cdots,N\}$,
\begin{equation}\label{E1}
\d X_t^i=b(X_t^i,\mathscr L_{X_t^i})\,\d t+\si\d W_t^i+\si_0(X_t^i)\,\d B_t^i
\end{equation}
 and
\begin{equation}\label{E2}
\d X^{i,N}_t=\tt b(X^{i,N}_{\theta_t},\tt{\mu}_{\bar\theta_t}^N)\,\d t+ \si\d W_t^i+\si_0(X^{i,N}_t)\,\d B_t^i
\end{equation}
with  the initial datum $(X_{[-r_0,0]}^{1,N},\cdots, X_{[-r_0,0]}^{N,N})$ for some $r_0>0$, where, for each $i\in\mathbb S_N,$ $X_{[-r_0,0]}^{i,N}$ is a $C([-r_0,0];\R^d)$-valued random variable.
In \eqref{E1} and \eqref{E2},  $b,\tilde b:\R^d\times\mathscr P(\R^d)\to\R^d$ with $\mathscr P(\R^d)$ being the set of probability measures on $\R^d$; $\si\in \R$ and  $\si_0:\R^d\to \R^d\otimes\R^m$; $\mathscr L_{X_t^i}$ stands for the law of $X_t^i$; the maps $\theta :[0,\8)\to[-r_0,\8)$  and $\bar\theta:[0,\8)\to[0,\8)$, where concrete  expressions will be specified later concerning respective settings; $\tt{\mu}_{{t}}^N:=\frac{1}{N}\sum_{j=1}^N\delta_{X_{ t}^{j,N}}$ means the empirical measure   associated with  particles $X_t^{1,N},\cdots, X_t^{N,N}$;
$W^1=(W_t^1)_{t\ge0},\cdots, W^N= (W_t^N)_{t\ge0}$ (resp. $B^1=(B_t^1)_{t\ge0},\cdots,B^N= (B_t^N)_{t\ge0}$) are mutually independent $d$-dimensional (resp. $m$-dimensional) Brownian motions supported on the same filtered probability space
$(\OO,\F,(\F_t)_{t\ge0},\P)$; Furthermore,  $(W^1,\cdots,W^N)$ is supposed to be independent of  $(B^1,\cdots,B^N)$.

Regarding the objects $( X_t^1,\cdots,X_t^N )_{t\ge0}$ and $( X_t^{1,N},\cdots,X_t^{N,N} )_{t\ge0}$ solving \eqref{E1} and \eqref{E2}, respectively,
the central goal in the present paper is to establish the  quantitative estimate:
\begin{equation}\label{E6}
\mathbb W_1\big(\mathscr L_{X_t^i},\mathscr L_{X_t^{i,N}}\big)\le \varphi(t,N),\quad t\ge0,~ i\in\mathbb S_N,
\end{equation}
where $\mathbb W_1$ denotes the $L^1$-Wasserstein distance  and $\varphi:[0,\8) \times [0,\8)\to(0,\8)$ is a decreasing function with respect to the first variable and the second argument, respectively.  For the explicit form of $\varphi$, we would like to refer to \eqref{B14} below for more details.
Hereinafter, we attempt  to elaborate why we focus on the frameworks \eqref{E1} and \eqref{E2}, and explore the uniform-in-time estimate \eqref{E6}. The interpretations will be expounded based on the   following three perspectives.

\subsubsection{Uniform-in-time propagation of chaos}
In \eqref{E2},   once we take $\tilde{b} =b $, $r_0=0$ and $\theta_t=\bar\theta_t=t$,  \eqref{E2} subsequently becomes
\begin{equation}\label{E5}
\d X^{i,N}_t=  b(X^{i,N}_t,\tt{\mu}_{{t}}^N)\,\d t+ \si\d W_t^i+\si_0( X^{i,N}_t)\,\d B_t^i.
\end{equation}
As we know, \eqref{E1} and \eqref{E5} are the respective non-interacting particle system and interacting particle system corresponding to  the following McKean-Vlasov SDE:
\begin{equation}\label{BH12}
\d X_t =b(X_t ,\mathscr L_{X_t})\d t+\si\d W_t+\si_0(X_t)\,\d B_t,
\end{equation}
where $(W_t)_{t\ge0}$ is a $d$-dimensional Brownian motion (a copy of $(W_t^i)_{t\ge0}$ for each $i\in\mathbb S_N$), which is independent of the $m$-dimensional Brownian motion  $(B_t)_{t\ge0}$ (a copy of $(B_t^i)_{t\ge0}$ for each $i\in\mathbb S_N$). Since  the landmark work \cite{Sznitman} due to Sznitman, the theory on propagation of chaos
in a finite-time horizon has achieved great advancements for various scenarios; see, for instance,  \cite{CD1,Huang} for McKean-Vlasov SDEs with regular coefficients, and \cite{BH} regarding McKean-Vlasov SDEs with irregular drifts or diffusions. Recently,   for weakly interacting mean-field particle systems with possibly non-convex
confinement and interaction potentials, the uniform-in-time propagation of chaos: for some constants $c,\lambda>0$ (independent of $t>0$ and $i\in\mathbb S_N$),
\begin{equation}\label{E7}
 \mathbb W_1\big(\mathscr L_{X_t^{i,\mu}}, \mathscr L_{X_t^{i,N,\nu}}\big)\le c\,\big(\e^{-\lambda t}\mathbb W_1(\mu,\nu)+N^{-\frac{1}{2}}  \big),\quad t>0,~i\in\mathbb S_N
\end{equation}
was established in the remarkable work \cite{DEGZ}, where  $\mathscr L_{X_t^{i,\mu}}$  and $\mathscr L_{X_t^{i,N,\nu}}$ stand respectively  for the distributions of $X_t^i$ and  $X_t^{i,N}$ with $\mathscr L_{X_0^i}=\mu$ and $\mathscr L_{X_0^{i,N}}=\nu$.  For more progresses on the uniform-in-time propagation of chaos concerning  Langevin dynamics with regular potentials and stochastic particle systems with mean-field singular interactions, we refer to \cite{GLM,GLMb,Schuh} and references within.
As an immediate  by-product of the quantitative estimate (see Theorem \ref{thm} below) derived in this paper,  the uniform-in-time propagation of chaos \eqref{E7}  will be reproduced  right away, where
the underlying drift terms might be singular with respect to the spatial variables (see Corollary \ref{cor} for more details).
The proceeding explanations  can be  viewed as one of our motivations to consider  \eqref{E1} and \eqref{E2}, and  study the estimate \eqref{E6}.
\subsubsection{Uniform-in-time distribution distance between an SDE and its delay version}
Consider a semi-linear SDE:
\begin{equation}\label{E3}
\d X_t= \beta(\alpha-X_t)\d t+\si \d W_t+\si_0(X_t)\,\d B_t,
\end{equation}
where $\alpha\in\R^d,\beta>0$, $\si\in\R$, $\si_0:\R^d\to \R^d\otimes\R^d $, $(W_t)_{t\ge0}$ and  $(B_t)_{t\ge0}$ are mutually independent $d$-dimensional Brownian motions.  In case of $\si_0(x)\equiv{\bf 0}_{d\times d}$ ($d\times d$ zero matrix), \eqref{E3} is a linear SDE solved by the   Ornstein–Uhlenbeck (O-U for abbreviation) process. As we know,
the O-U process
 has been applied considerably  in financial mathematics and the other related research fields. Whilst, in the real world, the price of an asset    or the evolution of population dynamics  is influenced inevitably by major events that took place. In turn, the viewpoint  above motivates us  to consider a memory-dependent version of \eqref{E3} which is described as follows:
\begin{equation}\label{E4}
\d Y_t=\beta(\alpha-Y_{t-r_0})\d t+\si \d W_t+\si_0(X_t)\,\d B_t,\quad t>0
\end{equation}
with the initial datum $Y_{[-r_0,0]}$. In \eqref{E4},
 $(\alpha,\beta, (W_t)_{t\ge0}, (B_t)_{t\ge0})$ is kept untouched as in \eqref{E3}, and the positive $r_0 $ is   the length of the time lag. Apparently,  \eqref{E3} and \eqref{E4} are fit into the frameworks  \eqref{E1} and \eqref{E2}  by setting $N=1$, $W^1_t=W_t$, and $B_t^1=B_t$, and taking  $\theta_t=t-r_0$ and $\tilde{b}=b$. Indeed, the quantity  $r_0 $ can be regarded as a perturbation. Intuitively speaking, the probability distance between   $\mathscr L_{X_t}$ and $\mathscr L_{Y_t}$ with $\mathscr L_{X_0}=\mathscr L_{Y_0}$ should be very small in case that the perturbation intensity $r_0 $ is tiny. So,  the issue on how to quantify  the  probability distance between $\mathscr L_{X_t}$ and $\mathscr L_{Y_t}$   encourages us to pursue  the topic \eqref{E6}. The above can be regarded as another
inspiration to implement the present work.

\subsubsection{Uniform-in-time discretization error bounds for stochastic algorithms}
Our third motivation arises from the long time analysis on stochastic algorithms for McKean-Vlasov SDEs, where the drifts need not to be uniformly dissipative with respect to the spatial variables. As   is known to all, the Euler-Maruyama (EM for short) scheme is the simplest and succinctest  method to discretize the McKean-Vlasov SDE \eqref{BH12} with $\si_0(x)\equiv{\bf 0}_{d\times d}$, that is,
\begin{equation}\label{EW}
\d X_t=b(X_t,\mathscr L_{X_t})\,\d t+\si\,\d W_t,\quad t>0.
\end{equation}
Nonetheless, the EM scheme works merely for  SDEs with coefficients of linear growth; see, for instance, \cite[Theorem 2.1]{HJK} and
\cite[Lemma 6.3]{MSH} for a theoretical support and a counterexample, respectively.
Based on this, plenty of  variants of the EM scheme   were proposed   to cope with  numerical approximations for  SDEs with   non-globally Lipschitz continuous coefficients. The finite-time strong convergence of the backward EM scheme, as a typical candidate of   EM's variant,  related to the McKean-Vlasov SDE \eqref{EW}:   for a step size $\delta>0,$
\begin{equation}\label{EW1}
\d X_t^{\delta,i,N}=b(X_{  t_\delta+\delta }^{\delta,i,N},\tilde\mu^{\delta,N}_{ t_\delta})\,\d t+\si\,\d W_t^i,\quad t>0,~i\in\mathbb S_N
\end{equation}
was explored in \cite{DES}, where $t_\delta:=\lfloor t/\delta\rfloor\delta$ with $\lfloor t/\delta\rfloor$ being the integer part of $t/\delta,$ and $\tilde\mu^{\delta,N}_t:=\frac{1}{N}\sum_{j=1}^N\delta_{X_t^{\delta,i,N}}$.
Transparently, \eqref{EW} and \eqref{EW1} with $r_0=0$ and $\si_0(x)\equiv{\bf0}_{d\times d}$ are included in \eqref{E1} and \eqref{E2}
by taking
$\tilde b=b$,   $\theta_t= t_\delta+\delta$,  and $\bar\theta_t=t_\delta$, separately.

Next, inspired by e.g. \cite{HJKa,Sabanis}, \cite{DES}  put forward  the   tamed EM scheme: for $\kk\in(0,1/2],$
\begin{equation}\label{EW3}
\d X_t^{\delta,i,N}=\frac{b(X_{ t_\delta}^{\delta,i,N},\tilde\mu^{\delta,N}_{ t_\delta})}{1+\delta^\kk|b(X_{ t_\delta}^{\delta,i,N},\tilde\mu^{\dd,N}_{ t_\delta})|}\,\d t+\si\,\d W_t^i,\quad t>0,~i\in\mathbb S_N
\end{equation}
to simulate   the McKean-Vlasov SDE \eqref{EW} in a finite time interval. Since, for a fixed step size $\delta>0,$ the modified drift is uniformly bounded,
the distribution of $(X_t^{\delta,i,N})_{t\ge0}$ solving the tamed EM scheme \eqref{EW3} is not adequate to approximate the distribution of $(X_t)_{t\ge0}$ determined  by \eqref{EW} in an infinite-time horizon.  Enlightened  by e.g. \cite{BDMS,LNSZ,LNSZb},  to derive a uniform-in-time estimate between the distributions of  the exact solution and the numerical counterpart, we construct  the following tamed EM scheme for the McKean-Vlasov  SDE \eqref{EW}:
\begin{equation}\label{EW2}
\d X_t^{\delta,i,N}=\frac{b(X_{ t_\delta}^{\delta,i,N},\tilde\mu^{\dd,N}_{ t_\delta})}{1+\delta^\kk\|\nn b(X_{ t_\delta}^{\delta,i,N},\tilde\mu^{\delta,N}_{ t_\delta})\|_{\rm HS}}\,\d t+\si\,\d W_t^i, \quad t>0,~i\in\mathbb S_N,
\end{equation}
kjn
where, for each fixed $\mu\in\mathscr P(\R^d)$,  $x\mapsto b(x,\mu)$ is a $C^1$-function, $\nn $ means the gradient operator with respect to the spatial variables, and
$\|\cdot\|_{\rm HS}$ stipulates the Hilbert-Schmidt norm. Compared \eqref{EW2} with \eqref{EW3}, the tamed drift in \eqref{EW2} might  not be bounded any more and  is at most of linear growth with respect to the spatial variables. Obviously, \eqref{E1} and \eqref{E2} with  $r_0=0$ and $\si_0(x)\equiv{\bf0}_{d\times d}$ can cover \eqref{EW} and \eqref{EW2} once we choose $\theta_t=\bar\theta_t=t_\delta$ and set
 $$
   \tilde b(x,\mu):=\frac{b(x,\mu)}{1+\delta^\kk\|\nn b(x,\mu)\|_{\rm HS}}
   .$$

No matter what the backward EM scheme \eqref{EW1} or the tamed EM algorithm \eqref{EW2}, the underlying time step is a uniform constant. In \cite{FG}, a refined EM scheme with an  adaptive step size was initiated to approximate SDEs with super-linear drifts.
In the spirit of \cite{FG}, \cite{RS} constructed  an adaptive EM scheme associated with \eqref{EW}, which is described as follows:
\begin{equation}\label{EW4}
   X_{t_{n+1}}^{\delta,i,N}= X_{t_n}^{\delta,i,N}+b\big(  X_{  t_n}^{\delta,i,N},\tilde\mu_{t_n}^{\delta,N}\big)h_n^\delta+\si\triangle W_{t_n}^i,\quad n\ge0,~i\in\mathbb S_N,
\end{equation}
where $t_{n+1}:=t_n+h_n^\delta$ for an  adaptive time step $h_n^\delta$ (see \eqref{T7} below for an alternative of $h_n^\delta$), and 
$\triangle W_{t_n}^i:=W_{t_{n+1}}^i-W_{t_{n }}^i$.
In contrast to \eqref{EW1} and \eqref{EW2}, the time step in \eqref{EW4} is not a  constant any more  but  an adaptive process,  which is determined by the current approximate solution.
Let $\underline{t} =\max\{t_n: t_n\le t\}$. Then,  the continuous version of \eqref{EW4} can be formulated as
\begin{equation}\label{EW5}
\d X_t^{\delta,i,N}=b\big(X_{\underline{t}}^{\delta, i,N},\tilde\mu_{\underline{t}}^{\delta,N}\big)\,\d t+\si\,\d W_t^i,\quad t>0,~i\in\mathbb S_N.
\end{equation}
Therefore, \eqref{EW} and \eqref{EW5} can be incorporated    into the frameworks of \eqref{E1} and \eqref{E2} by setting $\bar\theta_t=\theta_t=\underline{t}$, $\si_0(x)\equiv{\bf 0}_{d\times d}$, and $b=\widetilde{b}$.

With regard to the backward/tamed/adaptive EM scheme for classical SDEs and McKean-Vlasov SDEs, there is a huge amount of literature concerned with strong/weak convergence
in a finite-time interval; see  \cite{FG,HMS,KNRS,RS,Sabanis}, to name just a few. Meanwhile, there are still plenty of work handling long time behavior of numerical algorithms when (McKean-Vlasov) SDEs  involved are uniformly  dissipative with respect to the spatial variables; see  e.g. \cite{BMY,FG,NTDH,YL,YM} and references therein. In the aforementioned papers, the synchronous coupling   was employed to analyze the convergence property (in an infinite-time horizon) of the underlying algorithms.  Whereas, such an approach does not work any more to deal with the long time behavior of stochastic algorithms when (McKean-Vlasov) SDEs under investigation are not globally dissipative.  In the present work, as  another direct application of the main result (see Theorem \ref{thm}), concerning McKean-Vlasov SDEs,  we  quantify  uniform-in-time  estimates  on the probability distances between   laws of    exact solutions and   numerical solutions derived via    backward/tamed/adaptive EM schemes.  Once more, the elaborations above  urge us to work on the frameworks \eqref{E1} and \eqref{E2}, and conduct a further study on  \eqref{E6}.

\subsection{Main result}
Below, we assume that for $x\in\R^d$ and $\mu\in\mathscr P(\R^d)$,
\begin{align}\label{EP}
b(x,\mu)=b_1(x)+(b_0\ast \mu)(x)\quad \mbox{ with } (b_0\ast \mu)(x):=\int_{\R^d}b_0(x-y)\,\mu(\d y).
\end{align}
For such setting, we shall assume that the corresponding  SDEs \eqref{E1} and \eqref{E2} are strongly well-posed in order to establish a much more general result (i.e., Theorem \ref{thm}).  In the sequel, for SDEs under consideration, we shall present explicit conditions on the coefficients   to guarantee the   strong well-posedness. We further suppose that
\begin{enumerate}
 \item[$({\bf A}_1)$]$b_1(x)$ is continuous and locally bounded in $\R^d$ and
there exist constants $\ell_0\ge0$ and $\lambda>0$ such that for   $x,y\in\R^d$,
\begin{equation}\label{B1}
\<x-y,b_1(x)-b_1(y)\>\le |x-y|\phi(|x-y|)\mathds{1}_{\{|x-y|\le\ell_0\}}-\lambda|x-y|^2\mathds{1}_{\{|x-y|>\ell_0\}},
\end{equation}
where  $\phi:[0,\8)\to[0,\8)$ with $\phi(0)=0$ is increasing and continuous.
Moreover,  there exists a constant $K>0$ such that for all $x,y\in\R^d$,
\begin{equation} \label{B2}
|b_0(x)-b_0(y)|\le K|x-y|.
\end{equation}
\item[$({\bf A}_2)$] $\si\neq 0$, and there is a constant $L>0$ such that for all $x,y\in \R^d$,
\begin{equation}\label{B2-1}
	\|\si_0(x)-\si_0(y)\|_{\rm HS}^2\le L|x-y|^2.
\end{equation}
\end{enumerate}

To proceed, we make some comments on Assumptions $({\bf A}_1)$ and $({\bf A}_2)$.
\begin{remark}
In terms of \eqref{B1}, $b_1$ is dissipative merely in the long range. Particularly,  as revealed  in Corollary \ref{cor} below, \eqref{B1} allows the drift term $b_1$ to be singular. For example,
$b_1(x)=\bar b_1(x)+\tilde b_1(x)$, in which $\tilde b_1$ is   Dini continuous (see Assumption ({\bf H}) below for details), and  there exist constants  $\lambda_1',\lambda_2',\ell_0'>0$ such that for any $x,y\in\R^d$ ,
$$
\<x-y,\bar b_1 (x)-\bar b_1 (y)\>\le \lambda_1'|x-y|^2\mathds{1}_{\{|x-y|\le \ell_0'\}}-\lambda_2'|x-y|^2\mathds{1}_{\{|x-y|>\ell_0'\}}.
$$
Concerned with the diffusion terms, the additive part is set to be non-degenerate (which plays a crucial role in constructing the asymptotic coupling by reflection), and the multiplicative counterpart  might be degenerate.  Herein, we would like to stress that the additive intensity $\si$ considered in the present work  is a non-zero constant in lieu of a non-degenerate matrix  to write merely  the prerequisite
  \eqref{B1} and the asymptotic  coupling by reflection (see \eqref{B15} below for more details) in a simple way.
\end{remark}

For classical SDEs (with the same drifts and diffusions),
 it is enough to construct the reflection coupling  before the coupling time since  two SDEs will merge together afterwards due to  strong well-posedness (provided it exists). However, as far as two SDEs with different coefficients are concerned, the coupled processes can diverge once again even though they meet at the coupling time.  Therefore,  the classical reflection coupling approach no longer works    to estimate the probability distance between laws of solutions to SDEs with different coefficients.

 Inspired by \cite{Wang}, where gradient/H\"older estimates as well as exponential convergence were derived for   nonlinear monotone SPDEs, we shall design an asymptotic coupling by reflection to achieve the qualitative estimate \eqref{E6}.  To describe  the asymptotic coupling by reflection, we need to introduce some additional notations.
For $\vv\ge0$, let $h_\vv:[0,\8)\to[0,1]$ be a $C^1$-function satisfying
\begin{equation*}
h_\vv(r)=
\begin{cases}
0,\qquad 0\le r\le \vv,\\
1,\qquad r\ge 2\vv,
\end{cases}
\end{equation*}
and  $h_\vv^*:[0,\8)\to[0,1]$ be defined by $h_\vv^*(r)=\sqrt{1-h_\vv^2(r)}, r\ge0.$ Set
\begin{equation*}\label{B3}
\Pi(x):=I_{d\times d}-2{\bf e}(x)\otimes{\bf e}(x),\qquad x\in\R^d,
\end{equation*}
where $I_{d\times d}$ is the $d\times d$ identity matrix, ${\bf e}(x):=\frac{x}{|x|}\mathds{1}_{\{x\neq{\bf0}\}}$, and ${\bf e}(x)\otimes{\bf e}(x)$ means the tensor between ${\bf e}(x)$ and ${\bf e}(x)$. Thus, $\Pi $  defined above is an orthogonal matrix. Furthermore,
we shall assume that $W^{1,1}:=(W_t^{1,1})_{t\ge0}, \cdots, W^{1,N}:=(W_t^{1,N})_{t\ge0}$ (resp. $W^{2,1}:=(W_t^{2,1})_{t\ge0}, \cdots, W^{2,N}:=(W_t^{2,N})_{t\ge0}$) are mutually independent $d$-dimensional (resp. $m$-dimensional) Brownian motions carried on the same probability space as that of $B^1$, $\cdots$, $B^N$. In addition,  we  suppose that $(W^{1,1},\cdots, W^{1,N})$, $(W^{2,1},\cdots, W^{2,N})$ and $(B^{1},\cdots, B^{N})$ are mutually independent.

With the proceeding notations at hand,  we can write down the asymptotic coupling by reflection associated with \eqref{E1} and \eqref{E2}. More precisely,  we consider the following coupled interacting particle system:  for any $i\in\mathbb S_N$ and $t>0,$
\begin{equation}\label{B15}
 \begin{cases}
\d Y_t^{i,\vv}=b(Y_t^{i,\vv},\hat\mu_t^{i,\vv})\,\d t+\si h_\vv(|Z_t^{i,N,\vv}|)\d W_t^{1,i}+\si h_\vv^*(|Z_t^{i,N,\vv}|)\d W_t^{2,i}+\si_0(Y_t^{i,\vv})\, \d B_t^i, \\
\d Y_t^{i, N,\vv}= \tilde b(Y_{\theta_t}^{i,N,\vv},\tilde\mu_{\bar\theta_t}^{N,\vv})\,\d t+\si  \Pi( Z_t^{i,N,\vv} ) h_\vv (|Z_t^{i,N,\vv}|)\d W_t^{1,i}\\
 \qquad\quad\quad +\si h_\vv^*(|Z_t^{i,N,\vv}|)\d W_t^{2,i} + \si_0(Y_t^{i, N,\vv})\, \d B_t^i
\end{cases}
\end{equation}
with the initial condition $\big(Y_0^{i,\vv}, Y_{[-r_0,0]}^{i,N,\vv}\big)_{i\in\mathbb S_N}=\big(X_0^i,X_{[-r_0,0]}^{i,N}\big)_{i\in\mathbb S_N}$, which are i.i.d. random variables. In \eqref{B15}, the quantities $\hat\mu_t^{i,\vv},  Z_t^{i,N,\vv}$, and $\tilde\mu_t^{N,\vv}$ are defined respectively by
$$\hat\mu_t^{i,\vv} =\mathscr L_{Y_t^{i,\vv}},\quad Z_t^{i,N,\vv} =Y_t^{i,\vv}-Y_t^{i, N,\vv},\quad \tilde\mu_t^{N,\vv} =\frac{1}{N}\sum_{j=1}^N\delta_{Y_{t}^{j, N,\vv}}.$$

Now, we present some comments on the coupling constructed in \eqref{B15}. 
\begin{remark}
Note obviously  that the noise in \eqref{E1} include two parts, namely, the additive part and the multiplicative part.  In terms of \eqref{B15}, for the additive part, which is also non-degenerate, we adopt the asymptotic coupling by reflection; Whereas, for the multiplicative part (might be degenerate), we employ the synchronous coupling, which, in literature,  is also named as the coupling of marching soldiers. Moreover, we would like to emphasize that, for the construction of the asymptotic coupling by reflection, the drift term $b$ can be much more general rather than the form in \eqref{EP} as demonstrated in Lemma \ref{lemma1}. 
\end{remark}

Furthermore,
for the notational brevity, we  set for any $t\ge0$,
\begin{equation*}
{\bf X}_t^N:  =\big(X_t^{1},\cdots, X_t^{N}\big), \quad {\bf X}_t^{N,N}:  =\big(X_t^{1,N},\cdots, X_t^{N,N}\big)
\end{equation*}
and
\begin{equation*}
{\bf Y}_t^{N,\vv}:  =\big(Y_t^{1,\vv},\cdots, Y_t^{N,\vv}\big), \quad {\bf Y}_t^{N,N,\vv}:  =\big(Y_t^{1,N,\vv},\cdots, Y_t^{N,N,\vv}\big).
\end{equation*}
As claimed in Lemma \ref{lemma1} below, for any $\vv>0 $,   $( {\bf Y}_t^{N,\vv}, {\bf Y}_t^{N,N,\vv})_{t\ge0}$ is a coupling of $({\bf X}_t^N,{\bf X}_t^{N,N} )_{t\ge0}$.
Additionally,  for any $t\ge0$, $\mu \in\mathscr P(\R^d)$, and $\nu\in\mathscr P(\mathscr C)$ with $\mathscr C:=C([-r_0,0];\R^d)$ (i.e., the set of continuous $\R^d$-valued functions on $[-r_0,0]$),
denote   $\mu_t^i$ and $\nu_t^{i,N}$  by the laws of $X_t^i$ and $X_t^{i,N}$ with $\mathscr L_{X_0^i}=\mu$ and $\mathscr L_{X_{[-r_0,0]}^{i,N}}=\nu$, separately.

With the aid of previous  preliminaries, the quantitative estimate \eqref{E6} can be portrayed precisely as stated in the following theorem.

\begin{theorem}\label{thm}
Assume  Assumptions $({\bf A}_1)$ and $({\bf A}_2)$. Then, there are constants  $C,K^*,L^*,\lambda^*>0$ such that for any  $K\in[0,K^*]$, $L\in(0,L^*]$,   $\mu\in\mathscr P_{2}(\R^d)$, and  $\nu\in\mathscr P_1(\mathscr C)$,
\begin{equation}\label{B14}
\begin{split}
\mathbb W_1(\mu_t^i,\nu_t^{i,N})&\le C\bigg(\e^{-\lambda^* t}\mathbb W_1(\mu,\nu_0) +   N^{-\frac{1}{2}}\mathds{1}_{\{K>0\}} \\
&\qquad\quad+ \int_0^t\e^{-\lambda^* (t-s)}\E\big| b(X_s^{i, N},\tilde \mu_s^{N})-\tilde b(X_{\theta_s}^{i,N},\tilde\mu_{\bar\theta_s}^{N})\big| \,\d s\bigg),\quad t>0,\quad i\in\mathbb S_N,
\end{split}
\end{equation}
where  $\nu_0(\d x):=\nu(\{\eta\in\mathscr C\}: \eta_0\in\d x)$.
\end{theorem}

Before we proceed, let's make some remarks on Theorem \ref{thm}.
\begin{remark}
{\it Error bounds.} From \eqref{B14}, it is easy to see that $\mathbb W_1(\mu_t^i,\nu_t^{i,N})$ is dominated  by three terms, where the first term is concerned with the $L^1$-Wasserstein distance between the initial (projection) distributions with the decay prefactor $\e^{-\lambda^* t}$, the second one is related to the decay rate with respect to the particle number, and the third part involves the error among $b$ and $\tilde b$, which, in particular, embodies the dependence of the initial segment and the length of time lag when \eqref{E2} is an SDE  with memory; see Theorem \ref{thm3} below for more details. At first sight, the right hand side of \eqref{B14} is not elegant since the third term in the big parenthesis is not explicit. Nevertheless, the third term is much more tractable for applications we shall carry out.

 {\it Initial moments.}   When the drift term is written in the form \eqref{EP} and the associated initial distribution has a finite second-order moment, Theorem \ref{thm} shows that
the decay speed of $\mathbb W_1(\mu_t^i,\nu_t^{i,N})$ with respect to the particle number is $N^{-\frac{1}{2}}$. In some scenarios,
  the drift terms can be allowed to be much more general so the initial distribution necessitates merely a finite  lower-order moment.
Whereas, for this setting, the decay rate of $\mathbb W_1(\mu_t^i,\nu_t^{i,N})$ with respect to the particle number will be dependent on the dimension $d$ and become dramatically worse when,   in particular, \eqref{EEE} below is tackled by taking advantage of
\cite[Theorem 1]{Fournier}. Therefore, in the present work, we prefer the former framework rather than the latter one.

{\it Coupling construction.}
It is worthy to point out that, in \cite{Suzuki},  another kind of asymptotic coupling by reflection (which was called an approximate reflection coupling therein) was deployed to investigate bounds on the discretization error for Langevin dynamics, where the potential term is a $C^1$-function and is of linear growth.   Compared the asymptotic coupling by reflection in \cite{Suzuki} with \eqref{B15}, we find that the weak limit process of the coupled process constructed in \cite{Suzuki} is a coupling process while, for any $\vv>0 $,  the coupled process  $( {\bf Y}_t^{N,\vv}, {\bf Y}_t^{N,N,\vv})_{t\ge0}$  determined by  \eqref{B15} is a coupling process we desire. It is also worthy to emphasize that in \cite{Suzuki} a series of work on tightness need to be implemented in order to examine that the associated weak limit process is a coupling process. Therefore, the asymptotic coupling by reflection built in \eqref{B15} has its own advantages.

{\it Noise terms.} It seems to be slightly  weird that the noise term in \eqref{E1} encompasses two parts (i.e., the additive part and the multiplicative counterpart). Nevertheless, as long as the  diffusive term of the non-interacting particle system under investigation is multiplicative and non-degenerate, we can adopt the noise-decomposition trick (see e.g. \cite{PW}) so it can be decomposed equally in the sense of distribution into the format of \eqref{E1}. Based on point of view above, the framework \eqref{E1} can make  sense very well.

{\it Lipschitz constants.} In Theorem \ref{thm}, the constants $K^*$ and $L^*$ are the respective upper bounds of Lipschitz constants concerning $b_0$
%{\blue (with respect to the measure arguments) }
and $\si_0$. Via a close inspection of the proof for Theorem \ref{thm},  the   explicit forms concerning $K^*$ and $L^*$ can be tracked. As far as  McKean-Vlasov SDEs are concerned, the Lipschitz constant $K$ is generally  small otherwise the phase phenomenon can occur. Furthermore, the multiplicative intensity $\si_0$ is regarded as a perturbation of $\si$ so the corresponding Lipschitz constant $L$ should also be small provided that one wants to handle through  the asymptotic coupling by reflection the uniform-in-time estimate for SDEs with partially dissipative drifts.
\end{remark}

As an immediate by-product of Theorem \ref{thm}, we present the following statement, which is concerned with uniform-in-time propagation of chaos
for McKean-Vlasov SDEs, where one part of the drifts might be singular in the spatial variables.

\begin{corollary}\label{cor}
Assume  Assumptions $({\bf A}_1)$ and $({\bf A}_2)$. Then, there are constants  $C,K^*,L^*,\lambda^*>0$ such that for any  $K\in(0,K^*]$, $L\in[0,L^*]$,  $\mu\in\mathscr P_{2}(\R^d)$, $\nu\in\mathscr P_1(\R^d)$,    and $t\ge0,$
\begin{equation}\label{BH11}
\max_{i\in\mathbb S_N}\mathbb W_1(\mu_t^i,\nu_t^{i,N})\le C\big(\e^{-\lambda ^*t}\mathbb W_1(\mu,\nu) +   N^{-\frac{1}{2}}\big).
\end{equation}
 In particular, \eqref{BH11} holds true for the McKean-Vlasov SDE \eqref{BH12} with $b_1=\bar b_1+\tilde b_1 $
provided that
\begin{enumerate}
\item[$({\bf H})$]  $b_0$ satisfies \eqref{B2};
  $\bar b_1$ is Lipschitz in $\R^d$ and
satisfies \eqref{B1} with $\phi(r)=\lambda_0r  $ for some $\lambda_0>0$;  $\tilde b_1$ is  uniformly bounded and fulfills that
\begin{align*}
|\tilde b_1(x)-\tilde b_1(y)|\le \varphi(|x-y|),\quad x,y\in\R^d
\end{align*}
for some $\varphi\in\mathscr D$ with $\lim_{r\to\8}\varphi(r)/r=0$. Herein,
\begin{align*}
 \mathscr D:=\Big\{\varphi:\R_+\to\R_+\Big| \varphi(0)=0, \varphi \mbox{ is increasing, continuous, concave} \mbox{~~and~~} \int_0^1\frac{\varphi(s)}{s}\,\d s<\8\Big\} .
\end{align*}
\end{enumerate}
\end{corollary}

Below, we move forward to dwell on applications of Theorem \ref{thm}, and answer the remaining questions proposed in the introductory  subsections, one by one.

\subsection{Applications}

\subsubsection{Uniform-in-time distribution distance between an SDE and its delay version} For convenience, we first recall SDEs \eqref{E3} and \eqref{E4}. In this subsection, we focus on the following SDE:
\begin{equation}\label{E3E}
\d X_t= \beta(\alpha-X_t)\d t+\si \d W_t+\si_0(X_t)\,\d B_t,\quad t>0
\end{equation}
with the initial value $X_0\in L^1(\OO\to\R^d; \mathscr F_0,\P)$,
where $\alpha\in\R^d,\beta>0$, $\si\in\R$, $\si_0:\R^d\to \R^d\otimes\R^d $, which satisfies \eqref{B2-1};
$(W_t)_{t\ge0}$ and  $(B_t)_{t\ge0}$ are independent $d$-dimensional Brownian motions. Let  $(Y_t)_{t>0}$ be the delay version of $(X_t)_{t\ge0}$, which is determined by the  SDE with memory:
\begin{equation}\label{E4E}
\d Y_t=\beta(\alpha-Y_{t-r_0})\d t+\si \d W_t+\si_0(Y_t)\,\d B_t,\quad t>0
\end{equation}
with the initial value $Y_{[-r_0,0]}=\xi\in L^1(\OO\to\C; \mathscr F_0,\P)$ satisfying that for some  $C_\xi>0$,
\begin{align}\label{E22}
\E|\xi_t-\xi_s|\le C_\xi|t-s|, \quad t,s\in[-r_0,0].
\end{align}
Evidently, both \eqref{E3E} and \eqref{E4E} are strongly well-posed.

The following statement shows that the distributions $(\mathscr L_{X_t})_{t\ge0}$ and $(\mathscr L_{Y_t})_{t\ge0}$ associated with \eqref{E3E} and \eqref{E4E} respectively close to each other when the time lag $r_0$ approaches zero, and most importantly, provides a quantitative characterization upon the distribution deviation.

\begin{theorem}\label{thm3}
Assume that  $\si_0$ satisfies \eqref{B2-1} and suppose further   $\si\neq 0$ and $\beta>0$.
Then, there exist   constants $C^*,L^*,\lambda^*>0$ such that for all $L\in[0,L^*]$ and $t\ge0$,
\begin{align}\label{BH6}
\mathbb W_1\big(\mathscr L_{X_t},\mathscr L_{Y_t}\big)&\le C^*\Big(\e^{-\lambda^* t}\mathbb W_1\big(\mathscr L_{X_0},\mathscr L_{Y_0}\big) + (1+r_0) \big(1+  \E\|\xi\|_\8\big) r_0^{\frac{1}{2}}\Big).
\end{align}
\end{theorem}

\begin{remark}
Since the distribution of $(Y_t)_{t\ge0}$ is dependent on the segment $\xi\in L^1(\OO\to\C; \mathscr F_0,\P)$,   it is reasonable that the error bound on  the right hand side of \eqref{BH6} depends on $\E\|\xi\|_\8$ rather than $\E|\xi_0|$.
\end{remark}

\subsubsection{Uniform-in-time discretization error bounds for stochastic algorithms} \label{ssec3}
In this subsection, we focus on the McKean-Vlasov SDE \eqref{EW}, where the drift term is of super-linear  growth and dissipative in the long distance with respect to the spatial variables. As direct applications of Theorem \ref{thm},
we shall tackle  uniform-in-time discretization error bounds for the backward EM scheme, the tamed EM scheme, and the adaptive EM scheme, which are constructed   in \eqref{EW1}, \eqref{EW2}, and \eqref{EW5}, respectively.

In addition to Assumption  (${\bf A}_1$), we further need to suppose that the drift term $b_1$ is smooth and
 locally Lipschitz, which is stated precisely as below.
\begin{enumerate}
\item[$( {\bf A}_3)$] $x\mapsto b_1(x)$ is a $C^1$-function and  there exist constants $l^*\ge0, K^\star>0$ such that 
\begin{equation}\label{EE1}
 |b_1(x)-b_1(y)|\le K^{ \star}\big(1+|x|^{l^*}+|y|^{l^*}\big)|x-y|,\quad x,y\in\R^d.
\end{equation}
 \end{enumerate}

 Under Assumption (${\bf A}_1$) with $\phi(r)=\ll_0 r$ for some $\ll_0>0$, \eqref{EW} admits a unique strong solution for $X_0\in L^1(\OO\to\R^d;\mathscr F_0,\P)$; see, for example, \cite[Theorem 2.1]{Wanga} for more details.   Moreover,  note that  the discrete time version of \eqref{EW1}
 is indeed an implicit equation. Whereas,
 under \eqref{B1} with $\phi(r)=\ll_0 r$ and \eqref{B2}, the algorithm \eqref{EW1} is well defined
 as long as the step size $\delta\in(0,1/(2(\ll_0+K)))$; see, for instance, \cite[Lemma 3.4]{HMS} for related details.

With regard to the backward EM scheme, given in \eqref{EW1}, the long time error bound under the $L^1$-Wasserstein distance can be presented as follows. 	
\begin{theorem} \label{thm4}
Assume   $({\bf A}_1)$ with $\phi(r)=\ll_0 r$ for some $\ll_0>0$ and $\lambda>  2K $, and suppose further $\si \neq 0$ and $({\bf A}_3)$. Then, there exist constants $C^*,\ll ^*>0$ such that for all  $i\in \mathbb S_N$,   $t\ge0,$  and   $\delta\in(0,\delta^*]$,
\begin{equation}\label{F3-1}
\begin{split}
	\W_1\big(\mathscr L_{X_t^i},\mathscr L_{X_t^{\delta,i,N}}\big)\le C^* \bigg\{&\e^{-\ll ^* t}\W_1\big(\mathscr L_{X_0^i},\mathscr L_{X_0^{\delta,i,N}}\big)+\frac{1}{\sq{N}}\mathds{1}_{\{K>0\}}\\
 &+\Big(1+  \E\big|X_0^{\delta,i, N}\big|^{(l^*+1)^2}\Big)\delta^{\frac{1}{2}}\bigg\}
 \end{split}
\end{equation}
in case of  $\E|X_0^i|^2 <\8$ and $\E\big|X_0^{\delta,i, N}\big|^{(l^*+1)^2}<\8$. In the above,
\begin{align*}\label{HB5}
 \delta^*:=1\wedge \frac{1}{2(\ll_0+K)}\wedge \frac{3(\lambda-2K)}{4(\lfloor(1\vee l^*)(1+l^*)\rfloor+1)  (1+K)^{\lfloor(1\vee l^*)(1+l^*)\rfloor+1}}.
\end{align*}
\end{theorem}

\begin{remark}
Consider the SDE \eqref{EW} with $b(x,\mu)=b(x)$ (so \eqref{EW} is indeed a distribution-independent SDE). For this setting, via the Banach fixed point theorem, \cite[Corollary 2.3]{Eberle} implies that the corresponding solution process $(X_t)_{t\ge0}$ has a unique invariant probability measure, written as  $\pi.$ Once the the time-homogeneous Markov chain $(X^{\delta}_{k\delta})_{k\ge0}$ (which is determined by the backward EM scheme) possesses a unique invariant probability measure, denoted by  $\pi^{(\delta)}$,
Theorem \ref{thm4} enables us to deduce that $\mathbb W_1(\pi,\pi^{(\delta)})\le c\,\delta^{\frac{1}{2}}$ for some constant $c>0$. This reveals the quantitative estimate between the invariant probability measure and its numerical version for SDEs with partially dissipative drifts.
\end{remark}

Next, we apply Theorem \ref{thm} to the tamed EM scheme.
Furthermore, we suppose that
 \begin{enumerate}
\item[$({\bf A}_1')$]  
 there exist constants $\ll_{b_1},\hat\ll_{b_1},C_{b_1},\hat C_{b_1}>0$ such that for any $x\in \R^d$,
\begin{equation}\label{E*3}
\<x,b_1(x)\>\le -\ll_{b_1}|x|^2\cdot \|\nabla b_1(x)\|_{\rm HS}+C_{b_1},\qq|b_1(x)|\le \hat\lambda_{b_1}|x|\cdot\|\nabla b_1(x)\|_{\rm HS}+\hat C_{b_1}.
\end{equation}
Moreover, for some constant  $\alpha>0,$ there is   an  $ R_\alpha>0$ such that for $x\in\R^d$ with $|x|\ge R_\alpha,$
\begin{align}\label{**}
\|\nn b_1(x)\|_{\rm HS}\ge \alpha.
\end{align}
%\item[$({\bf A}_1'')$] There exist constants $\tilde \lambda_{b_1},l^*>0$ such that for all $x\in\R^d,$
%\begin{align*}
%\|\nn b_1(x)\|_{\rm HS}\le \tilde \lambda_{b_1}\big(1+|x|^{l^*}\big).
%\end{align*}
 \end{enumerate}

\begin{remark}
The second prerequisite in \eqref{E*3}
 %and Assumption $({\bf A}_1'')$ are
 is evidently satisfied when $b_1$ is of polynomial growth. Obviously, for $b_1(x)=x-x^\ell, x\in\R, $ with $\ell$ being an odd number $\ell\ge1$, the first technical condition in \eqref{E*3} and  the one in \eqref{**} are valid, separately. Moreover, in Assumption $({\bf A}_1')$
  %and $({\bf A}_1'')$
  , the gradient of $b_1$ is involved based on the construction of the tamed EM scheme presented  below.
\end{remark}

The tamed EM scheme associated with \eqref{EW}  is constructed as follows: for $\delta>0,$
\begin{equation}\label{BH14}
\mathrm{d} X_t^{\dd,i, N}=\big(b^\dd_1(X_{t_\dd}^{\dd,i, N})+(b_0\ast\tt\mu_{t_\dd}^{\dd,N})(X_{t_\dd}^{\dd,i, N})\big)\,\d t+\si  \mathrm{d} W_t^i,\qq i\in \mathbb S_N, \quad t > 0,
\end{equation}
where for any  $x\in\R^d,$
 \begin{equation*}\label{E*4}
	b_1^\dd(x):=\ff {b_1(x)}{1+\dd^{\frac{1}{2}} \|\nabla b_1(x)\|_{\rm HS}}.
\end{equation*}
Moreover, for brevity, we set for $\kk\ge 0$ and $\rho:=4K( K+ \hat\lambda_{b_1})$,
\begin{equation}\label{F5}
 \delta^*_\kk:=1\wedge  \alpha^{-\frac{1}{2}} \wedge \frac{\lambda_{b_1}^2}{\hat\lambda_{b_1}^4} \wedge
 \frac{\kk^2}{(2\alpha(2K+ \rho(1+1/\alpha)+\hat\lambda_{b_1}^2)^2 } \wedge \frac{1}{\kk/4+ K+\rho }.
\end{equation}

 Concerning the tamed EM scheme \eqref{BH14}, we have the following  discretization error bounds in an infinite-time horizon.
\begin{theorem}\label{thm5}
Assume Assumptions $({\bf A}_1)$ with $\phi(r)=\ll_0 r$ for some $\ll_0>0$ and $({\bf A}_1')$,   and suppose further 
$\kk:= \alpha\lambda_{b_1}-2 K>0$, $\ll>2 K$, $\si \neq 0$, as well as
$({\bf A}_3)$.  Then, there exist constants $C^*,\ll ^*>0$ such that for all      $\delta\in(0,\delta^*_\kk]$, $i\in \mathbb S_N$, and $t\ge0,$
\begin{equation}\label{F3}
	\W_1\big(\mathscr L_{X_t^i},\mathscr L_{X_t^{\delta,i,N}}\big)\le C^* \bigg\{\e^{-\ll ^* t}\W_1\big(\mathscr L_{X_0^i},\mathscr L_{X_0^{\delta,i,N}}\big)+\frac{1}{\sq{N}}\mathds{1}_{\{K>0\}} +\big(1+  \E\big|X_0^{\delta,i, N}\big|^{2l^*+1 }\big)\delta^{\frac{1}{2}}\bigg\}
\end{equation}
as long as   $\E|X_0^i|^{ 2}<\8$ and $\E\big|X_0^{\delta,i, N}\big|^{2l^*+1 }<\8$.
\end{theorem}

Finally, we  apply Theorem \ref{thm} to the adaptive EM scheme  \eqref{EW4} with the adaptive step size
\begin{align}\label{T7}
h^\delta_n:=\dd\min\bigg\{\frac{1}{1+|b(X_{ t_n}^{\dd,1,N},\tt\mu^{\delta,N}_{t_n})|^2},\cdots,\frac{1}{1+|b(X_{ t_n}^{\dd,N,N},\tt\mu^{\delta,N}_{t_n})|^2}\bigg\},\quad \delta\in(0,1).
\end{align}
Since, in  this paper, we are interested in the error analysis in an infinite-time horizon, the time grid $t_{n+1}=t_n+h_n^\delta$
should go  to infinity almost surely. This  can be  examined in  Lemma \ref{le} below.

 As far as the continuous-time version of \eqref{EW4}, defined accordingly in \eqref{EW5}, is concerned, the uniform discretization error bound is revealed as follows.

\begin{theorem} \label{thm6}
Assume   Assumptions $({\bf A}_1)$ with $\phi(r)=\ll_0 r$ and $\lambda>2K$, and suppose further $({\bf A}_3)$ and $\si \neq 0$.  Then, there exist constants $C^*,\ll ^*>0$ such that for all  $\delta\in(0,1),$  $i\in \mathbb S_N$, and $t\ge0,$
\begin{equation}\label{F8}
	\W_1\big(\mathscr L_{X_t^i},\mathscr L_{X_t^{\delta,i,N}}\big)\le C^* \bigg\{\e^{-\ll ^* t}\W_1\big(\mathscr L_{X_0^i},\mathscr L_{X_0^{\delta,i,N}}\big)+\frac{1}{\sq{N}}\mathds{1}_{\{K>0\}} +\big(1+  \E\big|X_0^{\delta,i, N}\big|^{  l^* }\big)\delta^{\frac{1}{2}}\bigg\}
\end{equation}
as long as   $\E|X_0^i|^{ 2}<\8$ and $\E\big|X_0^{\delta,i, N}\big|^{{1\vee l^*}}<\8$.
\end{theorem}

\begin{remark}
Since the discretization error is investigated under the $L^1$-Wasserstein distance, it is logical to require that the initial distribution for the algorithm under consideration has a finite moment of the first order. This indeed takes place in case of $l^*=0$ (which corresponds to the globally Lipschitz case for the drift involved) as demonstrated in Theorems \ref{thm4}, \ref{thm5} and \ref{thm6}. Nevertheless, concerning stochastic algorithms associated with McKean-Vlasov SDEs with drifts of super-linear growth with respect to spatial variables, it is quite natural to enhance the moment order  for  initial distributions. With contrast to the tamed/adaptive  EM scheme, higher order moments need to be imposed on the initial distributions for the backward EM scheme due to the fact that the tamed drift or adaptive step size can offset  growth of the original drift  in a certain sense.

Furthermore, we would like to say a few words on the noise term $\si_0$ in \eqref{E1}.  Once $\si_0$ does not vanish, the non-interacting particle system and the corresponding numerical version enjoy different noise terms. For this setup, the asymptotic coupling by reflection will no longer work
to investigate long time error bounds for stochastic algorithms. Based on this, in the present work, we focus merely on the additive noise case in lieu of the multiplicative noise setting.
\end{remark}

  The remainder part of this paper is organized as follows. Based on some preliminaries, in Section \ref{sec2}, we complete the proof of Theorem \ref{thm}
  by constructing an appropriate asymptotic coupling by reflection, and meanwhile finish the proof of Corollary \ref{cor}.
  Section \ref{sec3} is devoted to the proof of Theorem \ref{thm3}, where the variation-of-constants formula for semi-linear SDEs with memory plays a crucial role. In the  final section, with the aid of  uniform-in-time moment estimates for backward/tamed/adaptive EM schemes (where the underlying  proofs are rather tricky),
   we aim to implement proofs of Theorems \ref{thm4}, \ref{thm5} and \ref{thm6},   respectively.

\section{Proofs of Theorem \ref{thm} and Corollary \ref{cor}}\label{sec2}

Before the proof of Theorem \ref{thm}, we prepare some warm-up lemmas.

\begin{lemma}\label{lemma1}
Assume that the SDEs \eqref{E1} and \eqref{E2} are weakly well-posed. Then, for any $\vv>0,$ the path-valued processes ${\bf Y}^{N,\vv}$  and ${\bf Y}^{N,N,\vv}$  share  the common distributions as those of  ${\bf X}^N$ and  ${\bf X}^{N,N}$ on the path spaces $C([0,\8);\R^d)$ and $C([-r_0,\8);\R^d)$,  respectively.
\end{lemma}

\begin{proof}
 For   $i\in\mathbb S_N $ and $t\ge0,$ let
\begin{align*}
\bar W_t^{1,i}=\int_0^t\Pi(Z_s^{i,N,\vv}) \d W_s^{1,i}.
\end{align*}
With this shorthand notation,  the SDE  solved by  $ (Y_t^{i, N,\vv})_{i\in\mathbb S_N} $ can be reformulated as
\begin{equation}\label{B16}
\d Y_t^{i, N,\vv}=\tilde b\big(Y_{\theta_t}^{i,N,\vv},\tilde\mu_{\bar\theta_t}^{N,\vv}\big)\,\d t+\si  h_\vv (|Z_t^{i,N,\vv}|)\d \bar W_t^{1,i}+\si h_\vv^*(|Z_t^{i,N,\vv}|)\d W_t^{2,i}+\si_0(Y_t^{i, N,\vv})\,\d B_t^i
\end{equation}
with the initial value $Y_{[-r_0,0]}^{i,N,\vv}=X_{[-r_0,0]}^{i,N}$.
Observe that the SDE \eqref{B16}   has the same weak solution as that of the SDE solved by $(Y_t^{1,N,\vv},\cdots, Y_t^{N,N,\vv})_{t\ge0}$. Therefore, to complete the proof of Lemma \ref{lemma1},
 it is sufficient to show that the distributions of $({\bf Y}_t^{N,\vv})_{t\ge0}$ and $({\bf X}_t^N)_{t\ge0}$ are identical.

 Set  for   $i\in\mathbb S_N $ and $t\ge0,$
\begin{equation*}
 \tilde W_t^i:=\int_0^th_\vv(|Z_s^{i,N,\vv}|)\d W_s^{1,i}+  \int_0^th_\vv^*(|Z_s^{i,N,\vv}|)\d W_s^{2,i}.
\end{equation*}
Since $W^{1,i}$ is independent of $W^{2,i}$, besides   $h_\vv(r)^2+h_\vv^*(r)^2=1, r\ge0,$  Lévy's characterization  shows that  $ \tilde W^i $ is still a Brownian motion. Then, the SDE  solved by $(Y_t^{i,\vv})_{t\ge0}$ can be rewritten as an SDE driven by $\tilde W^i $. More precisely, we have
\begin{equation}\label{B9}
\d Y_t^{i,\vv}=b(Y_t^{i,\vv},\hat\mu_t^{i,\vv})\,\d t+\si  \d \tilde W_t^{ i}+\si_0(Y_t^{i,\vv})\,\d B_t^i,\quad i\in\mathbb S_N,~t>0.
\end{equation}

In order to prove that, for any $\vv>0,$ the distribution of $({\bf Y}_t^{N,\vv})_{t\ge0}$ is equal to that of
$({\bf X}_t^N)_{t\ge0}$, it remains to  verify that,  for any  $i\neq j, $ $\tilde W^{ i}$ and $\tilde W^{ j}$ are mutually independent. Indeed, by applying It\^o's formula, it follows that for  $u,v\in\R^d$, $i,j\in\mathbb S_N$ with $i\neq j,$ and $t>0,$
\begin{equation}\label{B4}
\begin{split}
\d \big(\<u, \tilde W^{ i}_t\>\<v,\tilde W^{ j}_t\>\big)&=\<v,\tilde W^{ j}_t\>\d \<u, \tilde W^{ i}_t\>+\<u, \tilde W^{ i}_t\>\d \<v,\tilde W^{ j}_t\>+\d [\<u, \tilde W^{ i}_t\>, \<v,\tilde W^{ j}_t\>]\\
&=\<v,\tilde W^{ j}_t\>\d \<u, \tilde W^{ i}_t\>+\<u, \tilde W^{ i}_t\>\d \<v,\tilde W^{ j}_t\>,
\end{split}
\end{equation}
where the second identity holds true due to the quadratic variation  $[\<u, \tilde W^{ i}_t\>, \<v,\tilde W^{ j}_t\>]=0$, which is valid  since
$W^{1,1}, \cdots, W^{1,N} $ (resp. $W^{2,1}, \cdots, W^{2,N}$) are independent and $(W^{1,1},\cdots, W^{1,N})$ is also independent of $(W^{2,1},\cdots, W^{2,N})$.  Obviously, \eqref{B4} manifests that  $(\<u, \tilde W^{ i}_t\>\<v,\tilde W^{ j}_t\>)_{t\ge0}$ is  a martingale.
This results in that    the covariance matrix $\E(  \tilde W_t^{ i}\otimes \tilde W_t^{ j}), i\neq j,$ is a $d\times d$ zero matrix.
Hence, we conclude that, for any  $i\neq j, $ $ \tilde W^{ i} $ and $\tilde W^{ j}$ are mutually independent.  Next, by following an
analogous procedure above, we deduce that $(\tilde W^{ 1}, \cdots,\tilde W^{ N}) $ is independent of $(B^1,\cdots, B^N).$
Subsequently, thanks to the weak uniqueness of \eqref{E1}, we conclude that  $({\bf X}_t^N)_{t\ge0}$ and $({\bf Y}_t^{N,\vv})_{t\ge0}$ possess  the same  distribution.
\end{proof}

\begin{lemma}\label{lemma2}
Assume  $({\bf A}_1)$ and $({\bf A}_2)$ with $\lambda> 2  K +\frac{1}{2}(1+(p-2)^+)L$  for   some     $p\ge 2 $. Then,  for any  $i\in\mathbb S_N,$ there exists a  constant  $C_p^* >0$ $($independent of  $i$$)$ such that
\begin{equation}\label{B10}
\sup_{t\ge0}\E\big|Y_t^{i,\vv}\big|^p\le C_p^{*}\big(1+  \E\big|Y_0^{i,\vv}\big|^p\big)
\end{equation}
as long as  $\E|Y_0^{i,\vv}|^p<\8.$
\end{lemma}

\begin{proof}
 For any $p\ge1$, let
\begin{equation*}
V_p(x)=(1+|x|^2)^{\frac{p}{2}},\quad x\in\R^d.
\end{equation*}
Performing a direct calculation shows that  for any $x\in\R^d,$
\begin{equation*}
\nn V_p(x)=p(1+|x|^2)^{\frac{p}{2}-1}x ~\mbox{ and }~ \nn^2 V_p(x)=p(1+|x|^2)^{\frac{p}{2}-1}I_{d\times d}+p(p-2)(1+|x|^2)^{\frac{p}{2}-2}(x\otimes x).
\end{equation*}
Next, by virtue of \eqref{B1} and \eqref{B2}, it follows that for all $x\in\R^d$ and $\mu\in\mathscr P_1(\R^d)$,
\begin{equation}\label{R*}
\<x,b(x,\mu)\>\le \ell_0\phi(\ell_0)+\lambda\ell_0^2- (\lambda-K)|x|^2
+|x|\big(|b_0({\bf0})|+|b_1({\bf0})|+K\mu(|\cdot|)\big).
\end{equation}
Then, applying  It\^o's formula to the SDE \eqref{B9}, we derive from \eqref{B2-1} that
\begin{align*}
&\d V_p(Y_t^{i,\vv})\\&=\Big(p\big(1+|Y_t^{i,\vv}|^2\big)^{\frac{p}{2}-1}\big\<Y_t^{i,\vv},b(Y_t^{i,\vv},\hat\mu_t^{i,\vv})\big\>\\
&\qquad+\frac{1}{2}p\si^2\big(1+|Y_t^{i,\vv}|^2\big)^{\frac{p}{2}-1}\big(d+(p-2)(1+|Y_t^{i,\vv}|^2)^{ -1}|Y_t^{i,\vv}|^2\big)\Big)\,\d t\\
&\qquad   +\ff 12 p\big (1+|Y_t^{i,\vv}|^2\big)^{\frac{p}{2}-1}\big( \|\si_0(Y_t^{i,\vv})\|^2_{\rm HS} +(p-2)(1+|Y_t^{i,\vv}|^2)^{ -1}|\si_0^*(Y_t^{i,\vv})Y_t^{i,\vv}|^2 \big)\,\d t +\d \bar M_t^{p,i}\\
&\le \Big(p\big(1+|Y_t^{i,\vv}|^2\big)^{\frac{p}{2}-1}\big(\ell_0\phi(\ell_0)+\lambda\ell_0^2- (\lambda-K)|Y_t^{i,\vv}|^2
+|Y_t^{i,\vv}|\big(|b_0({\bf0})|+|b_1({\bf0})|+K\E|Y_t^{i,\vv}|\big)\big) \\
&  \qquad+\frac{1}{2}p\big(1+|Y_t^{i,\vv}|^2\big)^{\frac{p}{2}-1}\big(\si^2 (d+(p-2)^+ )\\
&\qquad\qquad\qquad\quad\qquad\quad\quad\quad+\|\si_0(Y_t^{i,\vv})\|^2_{\rm HS}+(p-2)^+ \|\si_0(Y_t^{i,\vv})\|^2_{\rm op}\big)   \Big)\,\d t
+\d \bar M_t^{p,i},
\end{align*}
where  $(\bar M_t^{p,i})_{t\ge0}$ is a  martingale and $\|\cdot\|_{{\rm op}}$ means the operator norm.  For any $\kappa>0$, it can readily be seen from \eqref{B2-1} that there exists a constant $C_\kk>0$ such that
\begin{equation*}
\|\si_0(x)\|_{\rm HS}^2\le (L+\kappa)|x|^2+C_\kappa,\qquad x\in\R^d.
\end{equation*}
 Set $\lambda^\star:=\frac{1}{2}(\lambda- 2K-\frac{1}{2}(1+(p-2)^+)L)$, which is positive  due to $\lambda> 2K+\frac{1}{2}(1+(p-2)^+)L. $ Again, by applying It\^o's formula,
 there exists a positive constant $C_\star=C(\lambda^\star)$ such that
\begin{align*}
\d\big(\e^{p\lambda^\star t} V_p(Y_t^{i,\vv}) \big)
&\le \e^{p\lambda^\star t}\big(C_\star   -pK V_p(Y_t^{i,\vv}) +pK\big(1+|Y_t^{i,\vv}|^2\big)^{\frac{1}{2}(p-1)}  \E|Y_t^{i,\vv}| \big)\,\d t+ \e^{p\lambda^\star t} \d \bar M_t^{p,i}\\
 &\le \e^{p\lambda^\star t}\big(C_\star-  K V_p(Y_t^{i,\vv})
+ K   \E V_p(Y_t^{i,\vv})  \big)  \,\d t+\e^{p\lambda^\star t}  \d \bar M_t^{p,i},
\end{align*}
where we  utilized Young's inequality
in the first inequality and the second inequality, as well as Jensen's inequality in the second inequality. As a consequence, \eqref{B10}  is available
immediately.
\end{proof}

With preliminary Lemmas \ref{lemma1} and \ref{lemma2} at hand, we are in position to complete the
\begin{proof}[Proof of Theorem \ref{thm}]
Recall that, for all $t\ge0,$ $(Y_t^{i,\vv},Y_t^{i, N,\vv})_{i\in\mathbb S_N}$ solves    \eqref{B15} with   the initial value   $\big(Y_0^{i,\vv}, Y_{[-r_0,0]}^{i,N,\vv}\big)_{i\in\mathbb S_N}=\big(X_0^i,X_{[-r_0,0]}^{i,N}\big)_{i\in\mathbb S_N}$, which are i.i.d. random variables. For any $\mu\in\mathscr P_1(\R^d)$  and $\nu\in\mathscr P_1(\C)$,  in the following analysis, we choose  $(X_0^i,X_{[-r_0,0]}^{i,N})_{i\in\mathbb S_N} $ such that
\begin{equation}\label{B17}
 \mathbb W_1(\mu,\nu_0)=\E|X_0^i-X_0^{i,N}|, \quad i\in\mathbb S_N,
\end{equation}
in which $\nu_0(\d x):=\nu(\{\eta\in\mathscr C\}: \eta_0\in\d x)$.

Note that
\begin{align*}
\<I_{d\times d}-{\bf e}(x)\otimes {\bf e}(x), {\bf e}(x)\otimes {\bf e}(x)\big\>_{\rm HS}=0 \quad  ~\mbox{ and } ~  \quad  ({\bf e}(x)\otimes {\bf e}(x)){\bf e}(x)={\bf e}(x),\quad x\neq{\bf0}.
\end{align*}
Next, we shall fix the index $i \in \mathbb{S}_N$. For any $\eta \in(0,1]$, define the function $V_\eta$ by
$$
V_\eta(x)=(\eta+|x|^2)^{1 / 2}, \quad x \in \mathbb{R}^d,
$$
which is indeed a smooth approximation of the function $\mathbb{R}^d \ni x \mapsto|x|$. Applying Itô's formula and utilizing the facts:
$$
\nabla V_\eta(x)=\frac{x}{V_\eta(x)} \quad \text { and } \quad \nabla^2 V_\eta(x)=\frac{1}{V_\eta(x)} I_{d\times d}-\frac{x \otimes x}{V_\eta(x)^3},
$$
 it follows from the \eqref{B15} that 
 \begin{equation}\label{B-6}
 	\begin{aligned}
 		\d V_\eta(Z^{i,N,(\vv)}_t)\le & \ff1{V_\eta(Z^{i,N,(\vv)}_t)}\Psi_t^{i,N,\vv}\, \d t
 		+\ff1{V_\eta(Z^{i,N,(\vv)}_t)}\Big(\big\<Z_t^{i,N,\vv}, \big(\si_0(Y_t^{i,\vv})-\si_0(Y_{t}^{i,N,\vv})\big)\d B_t^i\big \>\\
 		&+2\si h_\vv(|Z_t^{i,N,\vv}|)\big\<Z_t^{i,N,\vv},\big({\bf e}(Z_t^{i,N,\vv})\otimes {\bf e}(Z_t^{i,N,\vv})\big)\d W_t^{1,i}\big\>\Big)\\
 		&+4\si^2\ff{h_\vv(|Z_t^{i,N,\vv}|)^2}{V_\eta(Z^{i,N,(\vv)}_t)}\big\<I_{d\times d}-\ff{Z_t^{i,N,\vv}\otimes Z_t^{i,N,\vv}}{|V_\eta(Z^{i,N,(\vv)}_t)|^2}
 		,{\bf e}(Z_t^{i,N,\vv})\otimes {\bf e}(Z_t^{i,N,\vv})\big\>_{\rm HS}\,\d t,
 	\end{aligned}
 \end{equation}
 where $\<\cdot,\cdot\>_{\rm HS}$ stands for the Hilbert-Schmidt inner product, and
\begin{align*}
\Psi_t^{i,N,\vv}:=\big\<Z_t^{i,N,\vv},b(Y_t^{i,\vv},\hat \mu_t^{i,\vv})-\tilde b(Y_{\theta_t}^{i,N,\vv},\tilde\mu_{\bar\theta_t}^{N,\vv})\big\>  +\ff 12 \|\si_0(Y_t^{i,\vv})-\si_0(Y_{t}^{i,N,\vv})\|^2_{\rm HS}.
\end{align*}
Note that for any $x\in \R^d$, when $\eta $ tends to 0,
\begin{align*}
\ff{x}{V_\eta (x)}\to  \ff x{|x|}\mathds 1_{\{x\neq {\bf 0\}}}	;\qquad\ff{h_\vv(x)^2}{V_\eta(x)}\big\<I_{d\times d}-\ff{x\otimes x}{|V_\eta(x)|^2}
 		,{\bf e}(x)\otimes {\bf e}(x)\big\>_{\rm HS}\le\ff{\eta}{{(\eta+\vv^2)}^{\ff32} }\to 0.
\end{align*}
Therefore, approaching $\eta\to0$ in \eqref{B-6}  leads to the estimates below:
\begin{equation}\label{B6}
\begin{split}
\d |Z_t^{i,N,\vv}|&
\le  
%\frac{1}{|Z_t^{i,N,\vv}|}\mathds{1}_{\{Z_t^{i,N,\vv}\neq {\bf0}\}}\big\{\Psi_t^{i,N,\vv} +4\si^2h_\vv(|Z_t^{i,N,\vv}|)^2\\
%&\quad\times\big\< I_{d\times d}-{\bf e}(Z_t^{i,N,\vv})\otimes {\bf e}(Z_t^{i,N,\vv}),   {\bf e}(Z_t^{i,N,\vv})\otimes {\bf e}(Z_t^{i,N,\vv})\big\>_{\rm HS}\big\}\,\d t\\
%&\quad+\frac{2\si}{|Z_t^{i,N,\vv}|} h_\vv(|Z_t^{i,N,\vv}|)\mathds{1}_{\{Z_t^{i,N,\vv}\neq {\bf0}\}}\big\<Z_t^{i,N,\vv},\big({\bf e}(Z_t^{i,N,\vv})\otimes {\bf e}(Z_t^{i,N,\vv})\big)\d W_t^{1,i}\big\>\\
%&\quad  +\frac{1}{|Z_t^{i,N,\vv}|}\mathds{1}_{\{Z_t^{i,N,\vv}\neq {\bf0}\}}\big\<Z_t^{i,N,\vv}, \big(\si_0(Y_t^{i,\vv})-\si_0(Y_{t}^{i,N,\vv})\big)\d B_t^i\big \>\\
\frac{1}{|Z_t^{i,N,\vv}|}\mathds{1}_{\{Z_t^{i,N,\vv}\neq {\bf0}\}} \Psi_t^{i,N,\vv}\,\d t +2\si h_\vv(|Z_t^{i,N,\vv}|)\mathds{1}_{\{Z_t^{i,N,\vv}\neq {\bf0}\}}\big\<{\bf e}(Z_t^{i,N,\vv}),\d W_t^{1,i}\big\>\\
 &\quad +\frac{1}{|Z_t^{i,N,\vv}|}\mathds{1}_{\{Z_t^{i,N,\vv}\neq {\bf0}\}}\big\<Z_t^{i,N,\vv}, \big(\si_0(Y_t^{i,\vv})-\si_0(Y_{t}^{i,N,\vv})\big)\d B_t^i\big \>.
\end{split}
\end{equation}

By splitting the quantity  $b(Y_t^{i,\vv},\hat\mu_{ t}^{i,\vv})-\tilde b(Y_{\theta_t}^{i,N,\vv},\tilde\mu_{\bar\theta_t}^{N,\vv})$ into three terms followed by taking \eqref{B1}, \eqref{B2}  as well as \eqref{B2-1} into consideration, we obtain that
\begin{equation}\label{B5}
\begin{split}
\Psi_t^{i,N,\vv}
&\le\<Z_t^{i,N,\vv},  (b_1(Y_t^{i,\vv})-b_1(Y_t^{i, N,\vv}) +\tt \psi^{i,N,\vv}_t
\>\\
&\quad +\<Z_t^{i,N,\vv},
 (b_0\ast\bar\mu_t^{N,\vv})(Y_t^{i, \vv})- (b_0\ast\tilde \mu_t^{N,\vv})(Y_t^{i, N,\vv})
  \>+\<Z_t^{i,N,\vv},  \psi^{i,N,\vv}_t\>+\ff 12 L|Z_t^{i,N,\vv}|^2\\
&\le \big(\phi(|Z_t^{i,N,\vv}|)+\lambda|Z_t^{i,N,\vv}|\big)|Z_t^{i,N,\vv}|\mathds{1}_{\{|Z_t^{i,N,\vv}|\le\ell_0\}}-(\lambda-K-L/2)|Z_t^{i,N,\vv}|^2  \\
&\quad+|Z_t^{i,N,\vv}|\big(K\mathbb W_1(\bar\mu_t^{N,\vv},\tilde \mu_t^{N,\vv})+  |\tt \psi^{i,N,\vv}_t|  + |\psi^{i,N,\vv}_t|\big),
\end{split}
\end{equation}
where $\bar \mu_t^{N,\vv}:=\frac{1}{N}\sum_{j=1}^N\delta_{Y_t^{j,\vv}}$,
\begin{align*}
 \tt \psi^{i,N,\vv}_t:=(b_0\ast\hat\mu_t^{i,\vv})(Y_t^{i, \vv})- (b_0\ast\bar \mu_t^{N,\vv})(Y_t^{i, \vv}) ~ \mbox{ and } ~ \psi^{i,N,\vv}_t:= b(Y_t^{i, N,\vv},\tilde \mu_t^{N,\vv})-\tilde b(Y_{\theta_t}^{i,N,\vv},\tilde\mu_{\bar\theta_t}^{N,\vv}).
\end{align*}
Inserting \eqref{B5} back into \eqref{B6} yields the estimate below:
\begin{equation}\label{E0}
\begin{split}
\d |Z_t^{i,N,\vv}|
&\le \mathds{1}_{\{Z_t^{i,N,\vv}\neq {\bf0}\}}  \big( (\phi(|Z_t^{i,N,\vv} |)+\lambda|Z_t^{i,N,\vv} |)\mathds{1}_{\{|Z_t^{i,N,\vv} |\le\ell_0\}}- (\lambda-K-L/2)|Z_t^{i,N,\vv} |\\
&\qquad\qquad\qquad+ K\mathbb W_1(\bar\mu_t^{N,\vv},\tilde \mu_t^{N,\vv}) + | \tt \psi^{i,N,\vv}_t|  +  |\psi^{i,N,\vv}_t|\big) \,\d t\\
&\quad+2\si h_\vv(|Z_t^{i,N,\vv}|)\mathds{1}_{\{Z_t^{i,N,\vv}\neq {\bf0}\}}\big\<{\bf e}(Z_t^{i,N,\vv}),\d W_t^{1,i}\big\>\\
&\quad+\frac{1}{|Z_t^{i,N,\vv}|}\mathds{1}_{\{Z_t^{i,N,\vv}\neq {\bf0}\}}\big\<Z_t^{i,N,\vv}, \big(\si_0(Y_t^{i,\vv})-\si_0(Y_{t}^{i,N,\vv})\big)\d B_t^i\big \> .
\end{split}
\end{equation}

Next, we define the $C^2$-function
\begin{align*}
f(r)=1-\e^{-c_1r}+c_2r,\quad r\ge0,
\end{align*}
where
\begin{align*}
c_1:=\frac{  2(\phi(\ell_0)+  (L +2K) \ell /2) }{ \si^2}   , \quad c_2:=c_1\e^{-c_1\ell_0}.
\end{align*}
Applying It\^o's formula, we derive from \eqref{E0} and $f''<0$ that
\begin{equation*}
\begin{split}
\d f(|Z_t^{i,N,\vv}|)&\le\big(\big(\varphi(|Z_t^{i,N,\vv}|)+2\si^2f''(|Z_t^{i,N,\vv}|)\big)h_\vv(|Z_t^{i,N,\vv}|)^2\mathds{1}_{\{Z_t^{i,N,\vv}\neq {\bf0}\}}+\varphi(|Z_t^{i,N,\vv}|)\big(1-h_\vv(|Z_t^{i,N,\vv}|)^2\big)\\
  &\qquad+f'( |Z_t^{i,N,\vv}|) \big(  K\mathbb W_1(\bar\mu_t^{N,\vv},\tilde \mu_t^{N,\vv}) + | \tt \psi^{i,N,\vv}_t|  +   |\psi^{i,N,\vv}_t| \big)\,\d t+\d M_t^{i,\vv}
\end{split}
\end{equation*}
for some martingale $(M_t^{i,\vv})_{t\ge0}$, where the function $\varphi :[0,\8)\to\R$ is defined by
\begin{align*}
\varphi (r) =  f'(r)\big((\phi(r)+\lambda r)\mathds{1}_{\{r\le\ell_0\}}- (\lambda -K-L/2)r  \big).
\end{align*}

Below, for the case $0\le r\le \ell_0$ and the case $r>\ell_0$,  we aim to verify respectively that
\begin{align}\label{E8}
\varphi^*(r):=\varphi(r) +2\si^2f''(r)\le -\lambda^*f(r),\quad r\ge0,
\end{align}
where for     $\lambda>K+L/2,$
\begin{align*}
\lambda^*:=\frac{((  2\phi(\ell_0)+  L\ell_0   )\wedge(  \lambda-K - L/2 ))c_2}{1-\e^{-c_1\ell_0}+c_2\ell_0}.
\end{align*}
By virtue of
\begin{align*}
f'(r)=c_1\e^{-c_1r}+c_2,\quad f''(r)=-c_1^2\e^{-c_1r},\quad r\ge0,
\end{align*}
it is easy to ready that for any $r\ge0,$
\begin{align*}
  \varphi^*(r) =(c_1\e^{-c_1r}+c_2)\big((\phi(r)+\lambda r)\mathds{1}_{\{r\le\ell_0\}}-(\lambda -K-L/2)r \big)-2\si^2c_1^2\e^{-c_1r}.
\end{align*}
For the case $0\le r\le \ell_0$,  in view of  $c_2=c_1\e^{-c_1\ell_0}\le c_1\e^{-c_1r}$ and the increasing property of $\phi$,
we find that
\begin{equation}\label{E9}
\begin{split}
\varphi^*(r)
 &\le-\big(2\si^2c_1^2-2c_1  (\phi(\ell_0) +  ( L+2K) \ell_0/2)\big)\e^{-c_1r} \\
 &\le - \si^2c_1^2 \e^{-c_1\ell_0} \\
&\le- \frac{c_1c_2\si^2}{1-\e^{-c_1\ell_0}+c_2\ell_0}f(r),
\end{split}
\end{equation}
where the second inequality is verifiable owing to
$
\si^2c_1^2=2c_1(\phi(\ell_0)+   L \ell_0/2),
$
and the last display is valid due to the fact that   $  r\mapsto f(r)$ is increasing on $[0,\8)$.
 On the other hand,  once $\lambda>K+L/2 $,  we derive that for $r>\ell_0,$
  \begin{equation}\label{B8}
\varphi^*(r)\le -c_2(\lambda-K - L/2)r     \le-\frac{c_2(\lambda-K  -L/2) }{1-\e^{-c_1\ell_0}+c_2\ell_0}  f(r),
\end{equation}
where the second inequality is provable since, for $\alpha>0,$   $  r\mapsto \frac{r}{1-\e^{-\alpha r}+r}$  is increasing on the interval $[0,\8)$ by taking the fundamental inequality: $1-\e^{-r}\ge r\e^{-r},r\ge0,$ into consideration. Consequently,
\eqref{E8} follows by combining \eqref{E9} with  \eqref{B8}.

In the sequel, invoking \eqref{E8} and taking  $c_2\le f'(r)\le c_1+c_2, r\ge0,$ and $f(0)=0$ into account
leads to
\begin{equation}\label{B11}
\begin{split}
\d f(|Z_t^{i,N,\vv}|)&\le\big(-\lambda^*  f(|Z_t^{i,N,\vv}|)h_\vv(|Z_t^{i,N,\vv}|)^2+ \varphi(|Z_t^{i,N,\vv}|)\big(1-h_\vv(|Z_t^{i,N,\vv}|)^2\big)       \\
&\quad\quad+  (c_1+c_2)     \big(   K\mathbb W_1(\bar\mu_t^{N,\vv},\tilde \mu_t^{N,\vv})+ |\tt \psi^{i,N,\vv}_t|  +|\psi^{i,N,\vv}_t|\big)  \big)\,\d t+\d M_t^{i,\vv}\\
&\le \d M_t^{i,\vv}+\Bigg(-\lambda^*  f(|Z_t^{i,N,\vv}|)+\big(\lambda^*f(|Z_t^{i,N,\vv}|)+\varphi(|Z_t^{i,N,\vv}|)\big)\big(1-h_\vv(|Z_t^{i,N,\vv}|)^2\big)  \\
&\quad\quad+  (c_1+c_2) \bigg(|\tt \psi^{i,N,\vv}_t|  +|\psi^{i,N,\vv}_t| +\frac{K}{c_2N}\sum_{j=1}^Nf(|Z_t^{j,N,\vv}|)\bigg)  \Bigg)\,\d t.
\end{split}
\end{equation}

Because  $(\tilde W^{ i})_{i\in\mathbb S_N}$  are independent, as demonstrated in the proof of Lemma \ref{lemma1}, and \eqref{B9} is strongly well-posed, $(Y^{i,\vv})_{i\in\mathbb S_N}$ are also independent. As a consequence, by following the line to drive \cite[(22), p.5396]{DEGZ},
we deduce from Lemma \ref{lemma2} with $p=2$ (in case of $\lambda>  2K +  L/2$)   that for some   constant $C_1>0$,
\begin{align}\label{EEE}
 \E|\tt \psi^{i,N,\vv}_t| \le C_1N^{-\frac{1}{2}}\mathds{1}_{\{K>0\}} ,\qquad t\ge0.
\end{align}
Substituting this estimate  into \eqref{B11}   yields for some constant $C_2>0,$
\begin{equation*}
\begin{split}
\frac{1}{N}\sum_{j=1}^N\d \E f(|Z_t^{j,N,\vv}|)\le\bigg(- \frac{\lambda^{**}}{N}\sum_{j=1}^N\E f(|Z_t^{j,N,\vv}|)  +    C_2\Lambda_t^{N,\vv}\bigg)\,\d t,
\end{split}
\end{equation*}
where
 $\lambda^{**}:=\lambda^*- K (1+c_1/c_2), $
and
\begin{align*}
\Lambda_t^{N,\vv}:=   N^{-\frac{1}{2}}+ \frac{1}{N}\sum_{j=1}^N \big(\lambda^*f(|Z_t^{j,N,\vv}|)+\varphi (|Z_t^{j,N,\vv}|)\big)\big(1-h_\vv(|Z_t^{j,N,\vv}|)^2\big) +     \frac{1}{N}\sum_{j=1}^N|\psi^{i,N,\vv}_t|.
\end{align*}

Obviously, there is a positive constant $K^* $  such that $\lambda^{**}>0$
for any $K\in[0,K^*] $. In what follows, we shall take $K\in[0,K^*] $ so that $\lambda^{**}>0$.
Whereafter,   an application of  Gronwall's inequality yields that
\begin{align*}
\frac{1}{N}\sum_{j=1}^N\E f(|Z_t^{j,N,\vv}|)&\le \e^{-\lambda^{**} t} \E f(|Z_0^{1,N,\vv}|)+C_2\int_0^t\e^{- \lambda^{**}(t-s)}\E\Lambda_s^{N,\vv} \,\d s,
\end{align*}
where we also explored the prerequisite that $(X_0^{i,\vv},X_{[-r_0,0]}^{i, N,\vv})_{i\in\mathbb S_N}=(Y_0^i,Y_{[-r_0,0]}^{i,N})_{i\in\mathbb S_N}$   are distributed identically. Once more, with the help of $c_2 r\le f(r)\le (c_1+c_2)r, r\ge0,$  along with \eqref{B17},
there is a constant $C_3>0$ such that
\begin{equation}\label{B13}
\frac{1}{N}\sum_{j=1}^N\E  |Z_t^{j,N,\vv}| \le C_3\bigg(\e^{-\lambda^{**} t}\mathbb W_1(\mu,\nu_0) + \int_0^t\e^{-\lambda^{**} (t-s)}\E\Lambda_s^{N,\vv} \,\d s\bigg).
\end{equation}
Next, according to the definition of $h_\vv$, in addition to  $f(0)=0$ and  $c_2\le f'(r)\le c_1+c_2$  for $r\ge0,$ it follows readily  that
\begin{align*}
\big(\lambda^*f(r)+\varphi(r)\big)\big(1-h_\vv(r)^2\big)
&\le 2(c_1+c_2)\big((\lambda^*+  K +L/2)r+\phi(r)\big)(1-h_\vv(r)\big)\\
&\le 2(c_1+c_2)\big(2(\lambda^*+  K +L/2)\vv+\phi(2\vv)\big):=\rho(\vv).
\end{align*}
Whence, we infer that for some constant $C_4>0,$
\begin{equation}\label{BH8}
\begin{split}
\frac{1}{N}\sum_{j=1}^N\E  |Z_t^{j,N,\vv}| \le C_4\Big(&\e^{-\lambda^{**} t}\mathbb W_1(\mu,\nu_0) +\rho(\vv)+N^{-\frac{1}{2}}\mathds{1}_{\{K>0\}}\\ &+\frac{1}{N} \sum_{j=1}^N \int_0^t\e^{-\lambda^{**} (t-s)}\E|\psi^{j,N,\vv}_s| \,\d s\Big).
\end{split}
\end{equation}

By the aid of Lemma \ref{lemma1},  $\mu_t^i=\mathscr L_{Y_t^{j,\vv}}$ and $\nu_t^{i,N}=\mathscr L_{Y_t^{j,N,\vv}}$ for each fixed $i\in\mathbb S_N $ and any $j\in\mathbb S_N.$  Moreover, recall that $(Y_0^{i,\vv}, Y_{[-r_0,0]}^{i,N,\vv})_{i\in\mathbb S_N}=(X_0^i,X_{[-r_0,0]}^{i,N})_{i\in\mathbb S_N}$  are independent and identically distributed.
Therefore, we derive that
\begin{equation*}
\mathbb W_1(\mu_t^i,\nu_t^{i,N})\le \frac{1}{N}\sum_{j=1}^N\E|Z_t^{j,N,\vv}|,\quad \forall\, i\in\mathbb S_N.
\end{equation*}
Thus, \eqref{BH8}  enables  us to derive that
\begin{equation*}
\begin{split}
\mathbb W_1(\mu_t^i,\nu_t^{i,N})&\le  C_4\Big(\e^{-\lambda^{**} t}\mathbb W_1(\mu,\nu_0) +\rho(\vv)+N^{-\frac{1}{2}}\mathds{1}_{\{K>0\}} +\frac{1}{N} \sum_{j=1}^N \int_0^t\e^{-\lambda^{**} (t-s)}\E|\psi^{j,N,\vv}_s| \,\d s\Big).
%&\le C_4\bigg(\e^{-\lambda^{**} t}\mathbb W_1(\mu,\nu_0) +\rho(\vv)+N^{-\frac{1}{2}}\\
%&\qquad\quad+\frac{1}{N} \sum_{j=1}^N \int_0^t\e^{-\lambda^{**} s}\E\big| b(X_s^{j, N},\tilde \mu_s^{N})-\tilde b(X_{\theta^1_s}^{j,N},\tilde\mu_{\theta^2_s}^{N})\big| \,\d s\bigg),
\end{split}
\end{equation*}
Subsequently, making use of Lemma \ref{lemma1} followed by  approaching $\vv\downarrow0$, and applying the prerequisite
  that $(X_{[-r_0,0]}^{i,N})_{i\in\mathbb S_N}$ are independent and identically distributed
yields that 
\begin{equation*}
\begin{split}
\mathbb W_1(\mu_t^i,\nu_t^{i,N})
&\le C_4\Big(\e^{-\lambda^{**} t}\mathbb W_1(\mu,\nu_0) +N^{-\frac{1}{2}}\mathds{1}_{\{K>0\}}  + \int_0^t\e^{-\lambda^{**} (t-s)}\E\big| \psi^{i,N}_s\big| \,\d s\Big),
\end{split}
\end{equation*}
 in which $$ \psi^{i,N}_s=b(X_s^{i, N},\tilde \mu_s^{N})-\tilde b(X_{\theta_s}^{i,N},\tilde\mu_{\bar\theta_s}^{N}).$$
Finally, the whole proof is complete    by choosing $K^*,L^*>0$ such that $\lambda^*,\lambda^{**}>0$ and $\lambda>2K+ L/2$ for all $K\in[0,K^*]$ and $L\in[0,L^*]$.
\end{proof}

\begin{proof}[Proof of Corollary \ref{cor}]
In terms of Theorem \ref{thm}, \eqref{BH11} follows immediately  by taking $r_0=0$, $\tilde b=b$, and $\theta_t=\bar\theta_t=t$. To show  \eqref{BH11} for the McKean-Vlasov SDE \eqref{BH12}
provided that Assumption ({\bf H}) is imposed, it is sufficient to prove the strong well-posedness and
check  respectively Assumptions $({\bf A}_1)$ and $({\bf A}_2)$.

Under Assumption ({\bf H}), \eqref{BH12} is strongly well-posed  once $X_0\in L^1(\OO\to\R^d;\mathscr F_0,\P)$; see, for instance, \cite[Theorem 4.1]{RW}. Trivially, Assumption (${\bf A}_2$) holds true.
Next, by virtue of \eqref{B1} with $\phi(r)=\lambda_0r$ and
 $b_1$ being replaced by $\bar b_1$, it follows readily that for $x,y\in\R^d$ and $\mu\in\mathscr P_1(\R^d)$,
\begin{align*}
\<x-y,b_1(x)-b_1(y)\>&=\<x-y, \bar b_1(x)-\bar b_1(y)\>+\<x-y,\tilde b_1(x)-\tilde b_1(y)\>\\
&\le |x-y|\varphi(|x-y|)+\big((\lambda_0+\lambda)\mathds{1}_{\{|x-y|\le\ell_0\}}-\lambda \big)|x-y|^2.
\end{align*}
Owing to  $\lim_{r\to\8}\frac{\varphi(r)}{r}=0$, there is an $r_0>\ell_0$ such that $\varphi(r)\le \frac{1}{2}\lambda r, r\ge r_0.$ Thus, we derive that
\begin{align*}
\<x-y,b_1(x)-b_1(y)\>
&\le |x-y|\big(\varphi(|x-y|)+ \lambda_0 |x-y|\big)\mathds{1}_{\{|x-y|\le r_0\}} -\frac{1}{2}\lambda|x-y|^2\mathds{1}_{\{|x-y|> r_0\}}.
\end{align*}
Therefore, \eqref{B1} is verifiable for $\phi(r)=\varphi(r)+\lambda_0r$, which obviously is increasing and satisfies $\phi(0)=0$.   This, together with Lipschitz property of $b_0,$ ensures Assumption (${\bf A}_1$).
\end{proof}

\section{Proof of Theorem \ref{thm3}}\label{sec3}
First of all, we demonstrate that the moment of the displacement for $(Y_t)_{t\ge0}$, determined by \eqref{E4E},
can be bounded by the length of time lag.
\begin{lemma}\label{lemma0}
Assume \eqref{B2-1} and $\beta>0$. Then, there exist constants $C^*,L^*>0$ such that for all $L\in[0,L^*]$ and $t\ge0,$
\begin{align}\label{BH3}
 \E\big| Y_t-Y_{t-r_0}\big|\le C^*  (1+r_0  )\big(1+  \E\|\xi\|_\8\big) r_0^{\frac{1}{2}}.
\end{align}
\end{lemma}

\begin{proof}
 To achieve \eqref{BH3}, we first show that there are constants  $C_0^*,L^* >0$ such that for all $L\in[0,L^*] $ and $t\ge0,$
 \begin{align}\label{BH2}
 \E\big(|Y_t|^2\big|\mathscr F_0\big)\le C_0^*\big(1+ (1+r_0)\|\xi\|^2_\8\big).
\end{align}
According to the variation-of-constants formula (see e.g. \cite[Theorem 3.1]{RRV}), we have for all $t\ge0,$
\begin{equation}\label{BH}
Y_t=\Gamma_t\xi_0-\beta\int_{-r_0}^0\Gamma_{t-r_0-s}\xi_s\,\d s+\beta\int_0^t\Gamma_{t-s}\alpha\,\d s+\int_0^t\Gamma_{t-s}\si\d W_s+\int_0^t\Gamma_{t-s}\si_0(Y_s)\d B_s,
\end{equation}
where $(\Gamma_t)_{t\ge0}$ solves the linear  ODE with memory
\begin{equation*}
\d \Gamma_t=-\beta\Gamma_{t-r_0}\,\d t,\quad t>0
\end{equation*}
with the initial condition  $\Gamma_0=I_{d\times d}$ and $\Gamma_r={\bf0}_{d\times d}, r\in[-r_0,0)$. By H\"older's inequality and  It\^o's isometry, it follows  that for any $\vv>0$ and $t\ge0$,
\begin{align*}
\E\big(|Y_t|^2\big|\mathscr F_0\big)&\le (1+\vv)\int_0^t\|\Gamma_{t-s}\|_{\rm op}^2\E\big(\|\si_0(Y_s)\|^2_{\rm HS}\big|\mathscr F_0\big)\d  s\\
&\quad+8(1+1/\vv)\bigg(\|\Gamma_{t }\|_{\rm op}^2 |\xi_0|^2+\beta^2 r_0 \int_{-r_0}^0\|\Gamma_{t-r_0-s}\|_{\rm op}^2 |\xi_s|^2\,\d s \\ &\qquad\qquad\qquad\quad\quad+\beta^2|\alpha|^2\Big(\int_0^t\|\Gamma_{ s}\|_{\rm op} \,\d s\Big)^2+\|\si\|_{\rm HS}^2\int_0^t\|\Gamma_{ s}\|_{\rm op}^2\d  s\bigg).
\end{align*}
Let
\begin{align*}
\lambda^\star =\sup\big\{\mbox{Re}(\lambda): \lambda\in\mathbb C,   \lambda+\beta\e^{-\lambda r_0}  =0\big\}.
\end{align*}
It is easy to see that $\lambda^\star<0$ thanks to $\beta>0.$
By invoking  \cite[Proposition A.1]{BWY},  for $\lambda_0\in(0,-\lambda^\star),$
there exists a constant $C_{\lambda_0}>0$ such that
\begin{align}\label{BH1}
\|\Gamma_t\|_{\rm op}\le C_{\lambda_0}\e^{-\lambda_0 t},\quad t\ge0.
\end{align}
This, together with \eqref{B2-1}, enables us to
  deduce  that  there exists a constant
 $C_\vv^\star>0$  such that
\begin{align*}
\E\big(|Y_t|^2\big|\mathscr F_0\big)&\le (1+\vv)^2C_{\lambda_0}^2L\int_0^t\e^{-2\lambda_0(t-s)}\E\big(|Y_s|^2\big|\mathscr F_0\big) \d  s +C_\vv^\star\big(1+ (1+r_0)\|\xi\|^2_\8\big).
\end{align*}
Subsequently, the Gronwall inequality yields that
\begin{align*}
\E\big(|Y_t|^2\big|\mathscr F_0\big)\le C_\vv^\star\big(1+ (1+r_0)\|\xi\|^2_\8\big)\bigg(1+(1+\vv)^2C_{\lambda_0}^2L\int_0^t \e^{-(2\lambda_0-(1+\vv)^2c_{\lambda_0}^2L)(t-s)}\,\d s\bigg).
\end{align*}
Since there exists an  $L^* >0$ such that $2\lambda_0- C_{\lambda_0}^2L>0$ so $2\lambda_0-(1+\vv)^2C_{\lambda_0}^2L>0$
for all $L\in[0,L^*]$, the assertion \eqref{BH2} follows directly.

By invoking  H\"older's inequality and It\^o's isometry,  we infer that for any $t\ge0$,
\begin{equation*}
\begin{split}
 \E\big(| Y_t-Y_{t-r_0}|\big|\mathscr F_0\big)&\le  \E\big(| Y_t-Y_{(t-r_0)^+}|\big|\mathscr F_0\big)+ |\xi_{t-r_0}-\xi_{0}| \mathds{1}_{[0,r_0]}(t)\\
 &\le |\xi_{t-r_0}-\xi_{0}| \mathds{1}_{[0,r_0]}(t)+ |\aa| \bb r_0+\beta\int_{(t-r_0)^+}^t\E\big(|Y_{s-r_0}|\big|\mathscr F_0\big)\,\d s\\
 &\quad+|\si|\E\big(| W_t-W_{(t-r_0)^+}|\big|\mathscr F_0\big)+ \E\Big(\Big|\int_{(t-r_0)^+}^t\si_0(Y_s)\,\d B_s\Big|\Big|\mathscr F_0\Big)\\
 &\le |\xi_{t-r_0}-\xi_{0}| \mathds{1}_{[0,r_0]}(t)+|\aa| \bb r_0+|\si|(dr_0)^{\frac{1}{2}}\\
 &\quad+ \beta\int_{(t-r_0)^+}^t\E\big(|Y_{s-r_0}|\big|\mathscr F_0\big)\,\d s+\bigg(\int_{(t-r_0)^+}^t\E\big(\|\si_0(Y_s)\|_{\rm HS}^2\big|\mathscr F_0\big)\,\d s\bigg)^{\frac{1}{2}}.
 \end{split}
\end{equation*}
Thus, by taking the Lipschitz property of $\si_0$, along with \eqref{E22} and \eqref{BH2},  into account, we find from H\"older's inequality that for some constant $c^\star>0$ and any $t\ge0,$
\begin{align*}
 \E\big(| Y_t-Y_{t-r_0}|\big|\mathscr F_0\big)\le |\xi_{t-r_0}-\xi_{0}| \mathds{1}_{[0,r_0]}(t)+c^\star(1+r_0)(1+\|\xi\|_\8)r_0^{\frac{1}{2}}.
\end{align*}
Consequently, \eqref{BH3} is available by taking advantage of \eqref{E22}.
\end{proof}

By invoking Lemma \ref{lemma0}, it is ready to carry out the
\begin{proof}[Proof of Theorem \ref{thm3}]
 By applying  Theorem \ref{thm},  there exist constants $C^*,\lambda^*>0$ such that
 \begin{equation*}
\mathbb W_1\big(\mathscr L_{X_t},\mathscr L_{Y_t}\big)\le C^*\bigg(\e^{-\lambda^* t}\mathbb W_1\big(\mathscr L_{X_0},\mathscr L_{Y_0}\big) + \int_0^t\e^{-\lambda^* (t-s)}\E\big| Y_s -Y_{s-r_0} \big| \,\d s\bigg).
\end{equation*}
Whence,  \eqref{BH6} is attainable by making use of \eqref{BH3} so the proof of Theorem \ref{thm3} is complete.
\end{proof}

\section{Proofs of Theorems \ref{thm4}, \ref{thm5} and \ref{thm6}}\label{sec4}

In this section, we aim to complete proofs of Theorems \ref{thm4}, \ref{thm5} and \ref{thm6}, respectively.

\subsection{Proof of Theorem \ref{thm4}}
To finish the proof of Theorem \ref{thm4}, we  first show that, for any $p\ge1,$ the $p$-th  moment of
 $(X_{n\delta}^{\dd,1,N}, \cdots, X_{n\delta}^{\dd,N,N})_{n\ge1}$, solving  \eqref{EW1},
is uniformly bounded.

\begin{lemma}\label{lemma3}
Assume   $({\bf A}_1)$ with $\phi(r)=\ll_0 r$ for some $\ll_0>0$ and $\lambda>  2K $. Then,  for any $p\ge1$ and $\delta\in(0,\delta_p]$ with
\begin{align}\label{HB5}
 \delta_{p }:=1\wedge \frac{1}{2(\ll_0+K)}\wedge \frac{3(\lambda-2K)}{4p(1+K)^{\lfloor p/2\rfloor+1}},
\end{align}
 there exists a  constant  $ C^\star_p>0$   such that for all $n\ge0$ and $i\in\mathbb S_N,$
\begin{equation}\label{EE8}
 \E\big|X_{n\delta}^{\dd,i,N}\big|^p\le \e^{-\lambda^*n\delta}\E\big|X_0^{\dd,i,N}\big|^p +C_p^*
 \end{equation}
in case of  $\E\big|X_0^{\dd,i,N}\big|^p<\8$, where
\begin{align*}
  \lambda^*:=\frac{3  (\lambda-   2 K)}{2(4+  3\lambda-  4 K )}.
\end{align*}

\end{lemma}
\begin{proof}
In the sequel, we would like to emphasize that all underlying constants are entirely unrelated to the step size, and let  $\mathscr F_0^N$ be the $\sigma$-algebra generated by $X_0^{1,N},\cdots,X_0^{N,N}$.
  For any even integer $p\ge2$ and $\delta\in(0,\delta_p]$, provided that there exists a constant $C_p^*>0$ such that
\begin{align}\label{HB4}
\E\big(|X_{n\dd}^{\dd,i,N}|^p\big|\mathscr F_0^N\big) &\le  \e^{- \lambda^*n\delta}|X_0^{\dd,i,N}|^p +C_p^*,\quad i\in\mathbb S_N,~ n\ge1,
\end{align}
then, from H\"older's inequality and the inequality: $(a+b)^\theta\le a^\theta+b^\theta$ for $a,b>0$ and $\theta\in(0,1]$, we obtain  that for any $p\in[1,2]$,
\begin{align*}
\E\big(|X_{n\dd}^{\dd,i,N}|^p\big|\mathscr F_0^N\big)&\le \Big(\E\big(|X_{n\dd}^{\dd,i,N}|^2\big|\mathscr F_0^N\big)\Big)^{\frac{p}{2}} \le  \e^{- \frac{p}{2}\lambda^*n\delta}|X_0^{\dd,i,N}|^p +\big(C_2^*\big)^{\frac{p}{2}}.
\end{align*}
and that for $p>2$ which is not an even number,
\begin{align*}
\E\big(\big|X_{n\dd}^{\dd,i,N}\big|^p\big|\mathscr F_0^N\big)&\le \Big(\E\big(\big|X_{n\dd}^{\dd,i,N}\big|^{2\lceil p/2\rceil) }\big|\mathscr F_0^N\big)\Big)^{\frac{p}{2\lceil p/2\rceil }} \\
&\le  \e^{-  \frac{p}{2\lceil p/2\rceil }\lambda^*n\delta}|X_0^{\dd,i,N}|^p +\big(C_{2\lceil p/2\rceil }^*\big)^{\frac{p}{2\lceil p/2\rceil }},
\end{align*}
where $\lceil\cdot\rceil$ means the ceiling function.
Hence, \eqref{HB4} is still valid for the other setting, where the  constants $\lambda_p$ and $C_p^*$  might be different accordingly.  Thus, \eqref{EE8} follows from \eqref{HB4} and the property of conditional expectation.

Let $\triangle  W_{n\dd}^i=W_{(n+1)\dd}^i-W_{n\dd}^i$.
Based on the preceding analysis, it remains to demonstrate \eqref{HB4}.
It is easy to see from \eqref{EW1} that
\begin{equation}\label{HB1}
\begin{split}
\big|X_{(n+1)\dd}^{\dd,i,N}\big|^2%&=\big|\big(X_{(n+1)\dd}^{\dd,i,N}-X_{n\dd}^{\dd,i,N}\big)+X_{n\dd}^{\dd,i,N}\big|^2\\
&=\big|X_{ n\dd}^{\dd,i,N}\big|^2-\big| X_{(n+1)\dd}^{\dd,i,N}-X_{n\dd}^{\dd,i,N} \big|^2+2\big\<X_{(n+1)\dd}^{\dd,i,N}-X_{n\dd}^{\dd,i,N},X_{(n+1)\dd}^{\dd,i,N}\big\>\\
&=\big|X_{ n\dd}^{\dd,i,N}\big|^2-\big| X_{(n+1)\dd}^{\dd,i,N}-X_{n\dd}^{\dd,i,N} \big|^2+2\dd\big\< b(X_{(n+1)\dd}^{\dd,i,N},\tt \mu^{\dd,N}_{n\dd}), X_{(n+1)\dd}^{\dd,i,N}\big\>\\
&\quad+2\si\big\<   \triangle  W_{n\dd}^i, X_{(n+1)\dd}^{\dd,i,N}-X_{n\dd}^{\dd,i,N}  \big\>+ 2\si\big\<   \triangle  W_{n\dd}^i, X_{n\dd}^{\dd,i,N} \big\>\\
&\le \big|X_{ n\dd}^{\dd,i,N}\big|^2+2\dd\big\< b(X_{(n+1)\dd}^{\dd,i,N},\tt \mu^{\dd,N}_{n\dd}), X_{(n+1)\dd}^{\dd,i,N}\big\>+\si^2\big|\triangle  W_{n\dd}^i\big|^2+2\si\big\<   \triangle  W_{n\dd}^i, X_{n\dd}^{\dd,i,N} \big\>,
\end{split}
\end{equation}
where  the last display is valid by making use of the fundamental inequality: $2ab\le a^2+b^2$
for any $a,b\in\R.$ Next, by means of  \eqref{R*}, it follows that
\begin{align}\label{T10}
\<x,b(x,\mu)\>
&\le C_0 -\frac{1}{4}(3\lambda-  4 K) |x|^2+\frac{1}{2}K\mu(|\cdot|^2),\qquad x\in\R^d, \mu\in\mathscr P_1(\R^d),
\end{align}
for some constant $C_0>0.$
Hence, we deduce from \eqref{HB1} and Jensen's inequality that
\begin{equation*}
\begin{split}
\big|X_{(n+1)\dd}^{\dd,i,N}\big|^2
&\le \big|X_{ n\dd}^{\dd,i,N}\big|^2+2\dd\Big(C_0 -\frac{1}{4}(3\lambda- 4K) |X_{(n+1)\dd}^{\dd,i,N}|^2+ \frac{1}{2}K \tt \mu^{\dd,N}_{n\dd}(|\cdot|^2 ) \Big)  \\
&\quad+\si^2\big|\triangle  W_{n\dd}^i\big|^2+2\si\big\<   \triangle  W_{n\dd}^i, X_{n\dd}^{\dd,i,N} \big\>.
\end{split}
\end{equation*}
This obviously implies that
\begin{equation}\label{HB3}
(1+\lambda_K\delta)\big|X_{(n+1)\dd}^{\dd,i,N}\big|^2 \le \big|X_{ n\dd}^{\dd,i,N}\big|^2+K\delta \tt \mu^{\dd,N}_{n\dd}(|\cdot|^2 )  +2C_0\delta+\si^2\big|\triangle  W_{n\dd}\big|^2+2\si\big\<   \triangle  W_{n\dd}, X_{n\dd}^{\dd,i,N} \big\>,
\end{equation}
in which $\lambda_K:=(3\lambda- 4K)/2 $.

 According to \eqref{HB3}, the binomial theorem gives that for any integer $p\ge1,$
 \begin{align*}
(1+\lambda_K\delta)^p\big|X_{(n+1)\dd}^{\dd,i,N}\big|^{2p}
&=\big(\big|X_{ n\dd}^{\dd,i,N}\big|^2+ K\delta  \tt \mu^{\dd,N}_{n\dd}(|\cdot|^2 )  \big)^{p}+ p\big(\big|X_{ n\dd}^{\dd,i,N}\big|^2+ K\delta  \tt \mu^{\dd,N}_{n\dd}(|\cdot|^2 )  \big)^{p-1}U_{n\delta}^i \\
&\quad+  \mathds 1_{\{p\ge2\}} \sum_{k=0}^{p-2}C_p^{k}\big(\big|X_{ n\dd}^{\dd,i,N}\big|^2+ K\delta  \tt \mu^{\dd,N}_{n\dd}(|\cdot|^2 )  \big)^{k}(U_{n\delta}^i)^{p-k}\\
&=:\Gamma_{n\delta}^{i}({\bf X}_{n\delta}^{\dd,N})+\hat\Gamma_{n\delta}^{i}({\bf X}_{n\delta}^{\dd,N})+\bar\Gamma_{n\delta}^{i}({\bf X}_{n\delta}^{\dd,N}),
\end{align*}
where
\begin{align}\label{T1}
	U_{n\delta}^i :=2C_0\delta+\si^2\big|\triangle  W_{n\dd}^i\big|^2+2\si\big\<   \triangle  W_{n\dd}^i, X_{n\dd}^{\dd,i,N} \big\>,\quad {\bf X}_{n\delta}^{\dd,N}:=\big(X_{ n\dd}^{\dd,1,N},\cdots,X_{ n\dd}^{\dd,N,N}\big).
\end{align}
Below, we attempt  to estimate the terms $\Gamma_{n\delta}^{i}$, $\hat\Gamma_{n\delta}^{i}$, as well as $\bar\Gamma_{n\delta}^{i}$, one by one. First of all, applying the binomial theorem and invoking Jensen's inequality and Young's inequality yields that
\begin{align*}
\Gamma_{n\delta}^{i}({\bf X}_{n\delta}^{\dd,N})&=\sum_{k=0}^pC_p^k \big|X_{ n\dd}^{\dd,i,N}\big|^{2k}\big( K\delta  \tt \mu^{\dd,N}_{n\dd}(|\cdot|^2 )\big)^{p-k}\\
&\le \sum_{k=0}^pC_p^k(K\delta)^{p-k}\bigg( \frac{k}{p}\big|X_{ n\dd}^{\dd,i,N}\big|^{2p}+\frac{p-k}{p}\tt \mu^{\dd,N}_{n\dd}(|\cdot|^{2p} )   \bigg).
\end{align*}
So, by utilizing the fact that $X_{ n\dd}^{\dd,i,N}$ and $X_{ n\dd}^{\dd,j,N}$ are identically distributed given $\mathscr F_0^N,$ we conclude
\begin{align*}
\E\big(\Gamma_{n\delta}^{i}({\bf X}_{n\delta}^{\dd,N})\big|\mathscr F_0^N\big)
&\le(1+K\delta)^p\E\big(|X_{ n\dd}^{\dd,i,N}|^{2p}\big|\mathscr F_0^N\big).
\end{align*}
Next, notice that
\begin{align*}
\E\big(\hat\Gamma_{n\delta}^{i}({\bf X}_{n\delta}^{\dd,N})\big|\mathscr F_0^N\big)=p\E\Big(\big(|X_{ n\dd}^{\dd,i,N}|^2+ K\delta  \tt \mu^{\dd,N}_{n\dd}(|\cdot|^2 )  \big)^{p-1}\big(2C_0\delta+\si^2\big|\triangle  W_{n\dd}^i\big|^2\big)\big|\mathscr F_0^N\Big).
\end{align*}
Whence, it is apparent that there exists a constant   $C_p>0$ such that
\begin{align*}
\E\big(\hat\Gamma_{n\delta}^{i}({\bf X}_{n\delta}^{\dd,N})\big|\mathscr F_0^N\big)\le \frac{1}{8}(\lambda-  2 K)p\delta\, \E\big(|X_{ n\dd}^{\dd,i,N}|^{2p}\big|\mathscr F_0^N\big)+C_p\delta.
\end{align*}
Once again, via Young's inequality and by utilizing the fact that $X_{ n\dd}^{\dd,i,N}$ and $X_{ n\dd}^{\dd,j,N}$ are identically distributed given $\mathscr F_0^N $, we infer that for some constant $C^\star_p>0$,
\begin{align*}
\E\big(\bar\Gamma_{n\delta}^{i}({\bf X}_{n\delta}^{\dd,N})\big|\mathscr F_0^N\big)\le \frac{1}{4}(\lambda-  2 K)p\delta\, \E\big(|X_{ n\dd}^{\dd,i,N}|^{2p}\big|\mathscr F_0^N\big)+C^\star_p\delta,
\end{align*}
where the underlying moment of the polynomial with respect to $|\triangle  W_{n\dd}^i|$ provides at least the order $\delta$.
Now, summing up the previous estimates on $\Gamma_{n\delta}^{i}$, $\hat\Gamma_{n\delta}^{i}$, and $\bar\Gamma_{n\delta}^{i}$, in addition to $(1+\lambda_K\delta)^p\ge 1+p\lambda_K\delta$,
enables us to derive that \begin{align*}
 (1+p\lambda_K\delta) \E\big(|X_{(n+1)\dd}^{\dd,i,N}|^{2p}\big|\mathscr F_0^N\big)
 &\le \big( (1+K\delta)^p +3(\lambda- 2 K)p\delta/8\big)\E\big(|X_{ n\dd}^{\dd,i,N}|^{2p}\big|\mathscr F_0^N\big)+C_p\dd.
\end{align*}
The mean value theorem, beside the definition of $\delta_p$ given in \eqref{HB5},   shows that for any $\delta\in(0,\delta_p)$,
\begin{align*}
(1+K\delta)^p&\le1+pK\delta+\frac{1}{2}(1+K)^{p-2}p(p-1)K^2\delta^2 \\
&\le 1+pK\delta+\frac{3}{8}( \lambda-  2  K)p\delta.
\end{align*}
 As a consequence, for
 \begin{align*}
 \lambda^{\star}_\delta:=\frac{1}{1+p\lambda_K\delta} \Big( 1+pK\delta +\frac{3}{4}( \lambda- 2K)p\delta\Big) \in(0,1)
 \end{align*}
 based on the prerequisite $\lambda>2K$, via an inductive argument,
we arrive at
\begin{align*}
\E\big(|X_{(n+1)\dd}^{\dd,i,N}|^{2p}\big|\mathscr F_0^N\big)
 \le\big(\lambda^{\star}_\delta\big)^{n} |X_0^{\dd,i,N}|^{2p} +\frac{ C^\star_p}{1-\lambda^{\star}_\delta }.
\end{align*}
Subsequently, by invoking the inequality: $a^r\le \e^{-(1-a)r}$ for $a,r>0$, we derive from $\delta\in(0,1)$ that
\begin{align*}
\E\big(|X_{(n+1)\dd}^{\dd,i,N}|^{2p}\big|\mathscr F_0^N\big)\le \e^{- \lambda_p^* n\delta}  |X_0^{\dd,i,N}|^{2p} +\frac{ C^\star_p }{\lambda_p^*},
\end{align*}
where
\begin{align*}
\lambda_p^*:=\frac{3p(\lambda-   2 K)}{2(2+p(3\lambda-  4 K))}.
\end{align*}
Consequently, the assertion \eqref{HB4} is available by noting that $p\mapsto \lambda_p^*$ is increasing.
\end{proof}

With the help of Lemma \ref{lemma3}, the proof of Theorem \ref{thm4} can be implemented.
  \begin{proof}[Proof of Theorem \ref{thm4}]
Below,  for nonnegative numbers $a,b$,  we use the shorthand notation $a\lesssim b$ if there exists a constant $c>0$ such that $a\le cb$.
Combining  \eqref{B2} with (${\bf A}_3$) and $\delta\in(0,1)$,  we derive from \eqref{EW1}   that    for any $t\ge0,$
\begin{equation}\label{F}
\begin{split}
\E\big(| X_t^{\delta,i, N}|^{2l^*}\big|\mathscr F_0^N\big)&\lesssim1+\E\big(|X_{t_\delta}^{\delta,i, N}|^{2l^*}\big|\mathscr F_0^N\big)+\E\big(\big|b(X_{  t_\delta+\delta }^{\delta,i,N},\tilde\mu^{\delta,N}_{ t_\delta})\big|^{2l^*}\big|\mathscr F_0^N\big)\\
&\lesssim1+\E\big(|X_{t_\delta}^{\delta,i, N}|^{2l^*}\big|\mathscr F_0^N\big)+\E\big(|X_{  t_\delta+\delta }^{\delta,i, N}|^{2l^*(l^*+1)}\big|\mathscr F_0^N\big)\\
&\quad+\frac{1}{N}\sum_{j=1}^N\E\big(|X_{t_\delta}^{\delta,j, N}|^{2l^*}\big|\mathscr F_0^N\big)\\
&\lesssim1+\E\big(|X_{t_\delta}^{\delta,i, N}|^{2l^*}\big|\mathscr F_0^N\big)+\E\big(|X_{  t_\delta+\delta }^{\delta,i, N}|^{2l^*(l^*+1)}\big|\mathscr F_0^N\big)\\
&\lesssim1+  |X_0^{\delta,i, N}|^{2l^*(l^*+1)},
\end{split}
\end{equation}
where in the penultimate inequality we used the fact that $(X_t^{\delta,j, N})_{j\in\mathbb S_N}$ are identically distributed given $\mathscr F_0^N$, and in the last display we applied Lemma \ref{lemma3}. Again, by taking \eqref{B2} and  (${\bf A}_3$) into consideration, along with $\delta\in(0,1)$, it follows from Lemma \ref{lemma3} that
\begin{align*}
\E\big(|X_t^{\delta,i, N}-X_{t_\delta}^{\delta,i, N}|^2\big|\mathscr F_0^N\big)&\lesssim \Big(1+\E\big(|X_{  t_\delta+\delta }^{\delta,i, N}|^{2(l^*+1)}\big|\mathscr F_0^N\big)+\E\big(|X_{  t_\delta }^{\delta,i, N}|^{2}\big|\mathscr F_0^N\big)\Big)\delta\\
&\lesssim\big(1+|X_0^{\delta,i, N}|^{2(l^*+1)}\big)\delta
\end{align*}
and subsequently from \eqref{F} and H\"older's inequality that
\begin{equation}\label{F2}
\begin{split}
&\E\big(| b(X_t^{\delta,i, N},\tilde \mu_t^{\delta,N})-b(X_{t_\delta+\delta}^{\delta,i,N},\tilde\mu_{t_\delta}^{\delta,N})|\big|\mathscr F_0^N\big)\\
&\lesssim  \Big(1+\big(\E\big(| X_t^{\delta,i, N}|^{2l^*}\big|\mathscr F_0^N\big)\big)^{\frac{1}{2}}+\big(\E\big(|X_{  t_\delta+\delta }^{\delta,i, N}|^{2l^*}\big|\mathscr F_0^N\big)\big)^{\frac{1}{2}}\Big)\\
&\quad\times\sup_{s-t\le2\delta}\big(\E\big(|X_t^{\delta,i, N}-X_s^{\delta,i, N}|^2\big|\mathscr F_0^N\big)\big)^{\frac{1}{2}}\\
&\quad+\frac{1}{N}\sum_{j=1}^N\E\big(|X_t^{\delta,j, N}-X_{t_\delta}^{\delta,j, N}|\big|\mathscr F_0^N\big)\\
&\lesssim \big(1+  \big|X_0^{\delta,i, N}\big|^{(l^*+1)^2}\big)\delta^{\frac{1}{2}}.
\end{split}
\end{equation}
Next, applying Theorem \ref{thm} with $r_0=0,$
 $\theta_t=t_{\dd}+\dd$, $\bar\theta_t=t_\dd$, and $\tilde b=b$
yields that for some constant $\lambda^*>0,$
\begin{equation*}
\begin{split}
\mathbb W_1\big(\mathscr L_{X_t^i},\mathscr L_{X_{t}^{\dd,i,N}}\big)&\lesssim  \e^{-\lambda^* t}\mathbb W_1\big(\mathscr L_{X_0^i},\mathscr L_{X_0^{\dd,i,N}}\big) +   N^{-\frac{1}{2}} \\
&\qquad\quad+ \int_0^t\e^{-\lambda^* (t-s)}\E\big| b(X_s^{\delta,i, N},\tilde \mu_s^{\delta,N})-b(X_{s_\delta+\delta}^{\delta,i,N},\tilde\mu_{s_\delta}^{\delta,N})\big| \,\d s.
\end{split}
\end{equation*}
Whence, the desired assertion \eqref{F3-1} follows from \eqref{F2}.
\end{proof}

\subsection{Proof of Theorem \ref{thm5}}

By following the line to handle the proof of Theorem \ref{thm4}, it is necessary to verify that $(X_{n\delta}^{\dd,1,N}, \cdots, X_{n\delta}^{\dd,N,N})_{n\ge1}$, determined by \eqref{BH14}, is uniformly bounded in the moment sense.
\begin{lemma}\label{lemma4}
Assume Assumptions $({\bf A}_1)$ with $\phi(r)=\ll_0 r$ for some $\ll_0>0$ and $({\bf A}_1')$,   and suppose further
$\kk:= \alpha\lambda_{b_1}-2 K>0$ and $\ll>2 K$. Then,  for any $p\ge1$ and $\delta\in(0,\delta^*_\kk]$ with $\delta^*_\kk$ being defined in \eqref{F5},
  there is a constant $C^\star_p>0$ such that
\begin{equation}\label{*---}
\E\big|X_{n\delta}^{\dd,i,N}\big|^p\le C^\star_p\big(1+\E\big|X_0^{\dd,i,N}\big|^p\big),\quad n\ge1.
\end{equation}
\end{lemma}

\begin{proof}
By tracing  the proof of Lemma \ref{lemma3}, it suffices to verify that, for any integer $p\ge3 $ and $\delta\in(0,\delta_\kk^*],$ there exists a constant $C^{\star\star}_p>0 $ such that
\begin{align}\label{T2}
\E\big(|X_{(n+1)\delta}^{\dd,i,N}|^{2p}\big|\mathscr F_0^N\big)\le (1- \dd/16)\E\big(|X_{n\delta}^{\dd,i,N}|^{2p}\big|\mathscr F_0^N\big)+C^{\star\star}_p\dd
\end{align}
in order to achieve \eqref{*---}.

From \eqref{BH14}, it can  be seen readily that
\begin{align*}
\big|X_{(n+1)\delta}^{\dd,i,N}\big|^2
&=\big|X_{n\delta}^{\dd,i,N}\big|^2+\big(2\<X_{n\delta}^{\dd,i,N},b_1^\delta(X_{n\delta}^{\dd,i,N})+\big(b_0\ast\tt\mu_{n\dd}^{\dd,N}\big)(X_{n\dd}^{\dd,i, N})\>
+\delta\big|b_1^\delta(X_{n\delta}^{\dd,i,N})\big|^2\big)\delta\\
&\quad+ \big\<\big(b_0\ast\tt\mu_{n\dd}^{\dd,N}\big)(X_{n\dd}^{\dd,i, N})+2b_1^\delta(X_{n\delta}^{\dd,i,N}),\big(b_0\ast\tt\mu_{n\dd}^{\dd,N}\big)(X_{n\dd}^{\dd,i, N})\big\>\big)\delta^2\\
&\quad+\big(\si^2\big|\triangle W_{n\delta}^i\big|^2 +2\si\big\<X_{n\delta}^{\dd,i,N}+b_1^\delta(X_{n\delta}^{\dd,i,N})\delta+\big(b_0\ast\tt\mu_{n\dd}^{\dd,N}\big)(X_{n\dd}^{\dd,i, N})\delta, \triangle W_{n\delta}^i\big\>\big)\\
&=:\big|X_{n\delta}^{\dd,i,N}\big|^2+\Lambda^{i,\delta}({\bf X}_{n\delta}^{\delta,N})\delta +\hat\Lambda^{i,\dd}({\bf X}_{n\delta}^{\delta,N})\delta^2+\bar\Lambda^{i,\delta}({\bf X}_{n\delta}^{\delta,N}),
\end{align*}
where ${\bf X}_{n\delta}^{\delta,N}$ was defined as in \eqref{T1}.

For any ${\bf x}:=(x_1,\cdots,x_N)\in(\R^d)^N$, let $\mu_{{\bf x}}^N=\frac{1}{N}\sum_{j=1}^N\delta_{x_j}$.
By invoking  \eqref{B2} and \eqref{E*3}, we derive   that  for any $\delta\in(0,\delta^*_\kk],$
\begin{align*}
\Lambda^{i,\delta}({\bf x})
&\le\frac{1}{1+\delta^{\ff12}\|\nn b_1(x_i)\|_{\rm HS}}\bigg(2\<x_i,b_1(x_i)\>+\frac{\delta|b_1(x_i)|^2}{1+\delta^{\ff12}\|\nn b_1(x_i)\|_{\rm HS}}\bigg) +2|x_i|\,\mu_{{\bf x}}^N(|b_0(x_i-\cdot)|)\\
&\le  - \frac{2\|\nn b_1(x_i)\|_{\rm HS}}{1+\delta^{\ff12}\|\nn b_1(x_i)\|_{\rm HS}}\big(  \ll_{b_1} -\hat\lambda_{b_1}^2\delta^{{\ff12}}\big)|x_i|^2+2\big(C_{b_1}+\hat C_{b_1}^2\big) \\
&\quad+ K\big(3|x_i|^2 + \mu_{{\bf x}}^N(|\cdot|^2)\big) +2|x_i|\cdot|b_0({\bf0})|
\end{align*}
and that there exists a   constant $C_0^*>0$  such that
\begin{align*}
\hat \Lambda^{i,\delta}({\bf x})&=\big|\mu_{{\bf x}}^N(b_0(x_i-\cdot)) \big|^2+\frac{2\<b_1(x_i),\mu_{{\bf x}}^N(b_0(x_i-\cdot))\>}{1+\delta^{\ff 12}\|\nn b_1(x_i)\|_{\rm HS}}  \\
&\le 4K^2\big(|x_i|^2+\mu_{{\bf x}}^N(|\cdot|^2)\big)+2|b_0({\bf0})|^2\\
&\quad+2 \big(\hat\lambda_{b_1}\delta^{-^{\ff12}} |x_i|+\hat C_{b_1}\big)\big(  K |x_i| +K\mu_{{\bf x}}^N(|\cdot|) + |b_0({{\bf0}})|  \big)\\
&\le 4K\big( K +  \hat\lambda_{b_1}  \delta^{-^{\ff12}}\big)\big(|x_i|^2+\mu_{{\bf x}}^N(|\cdot|^2)\big) +C_0^*(1+\delta^{-^{\ff12}}).
\end{align*}
Thus, combining  \eqref{**} with $\delta\in(0,1)$, in addition to the local boundedness of $\nn b_1$,  yields that for some constant $C_1^*>0$,
 \begin{align*}
\Lambda^{i,\delta}({\bf x})\delta +\hat\Lambda^{i,\dd}({\bf x})\delta^2&\le \bigg(- \frac{2\alpha}{1+\delta^{\ff12}\alpha}\Big(  \ll_{b_1}-\frac{1}{2\alpha}\big(3 K+\rho\delta^{\frac{1}{2}} \big)(1+\delta^{\ff12}\alpha)-\hat\lambda_{b_1}^2\delta^{\frac{1}{2}}  - \frac{\kk}{4\alpha} \Big)|x_i|^2\\
&\quad  +(K+\rho\delta^{\frac{1}{2}} )\mu_{{\bf x}}^N(|\cdot|^2)\bigg)\delta +C_1^* \delta\\
  &\le\bigg(- \frac{2\alpha}{1+\delta^{\ff12}\alpha}\Big(  \frac{3\kk}{4\alpha}-\big(2 K+\rho(1+1/\alpha)+\hat\lambda_{b_1}^2\big)\delta^{\frac{1}{2}} \Big)|x_i|^2\\
  &\quad+(K+\rho\delta^{\frac{1}{2}} )\big(\mu_{{\bf x}}^N(|\cdot|^2)-|x_i|^2\big)\bigg)\delta +C_1^* \delta,
\end{align*}
 where $\rho:=4K(K + \hat\lambda_{b_1}  )$.
In terms of the definition of $\delta^*_\kk$, we  right now have for any $\delta\in(0,\delta^*_\kk],$
\begin{align*}
 (2K+\rho(1+1/\alpha)+\hat\lambda_{b_1}^2)\delta^{\frac{1}{2}}\le\frac{\kk}{2\alpha}.
\end{align*}
Whence, owing to $\delta^{\ff12}\alpha\in(0,1)$ for $\delta\in(0,\delta^*_\kk],$  we infer   that
 \begin{align*}
\Lambda^{i,\delta}({\bf x})\delta +\hat\Lambda^i({\bf x})\delta^2
  &\le\big(-  (  \kk/4 +\beta_\delta)   |x_i|^2 +\beta_\delta\mu_{{\bf x}}^N(|\cdot|^2) \big)\delta +C_1^* \delta,
\end{align*}
where $\beta_\delta:=K+\rho\delta^{\frac{1}{2}}$. Whereafter,  the  preceding estimate    enables us to deduce that
\begin{align}\label{*-}
\big|X_{(n+1)\delta}^{\dd,i,N}\big|^2
\le \big(1-(  \kk/4 +\beta_\delta)\delta\big)\big|X_{n\delta}^{\dd,i,N}\big|^2+\beta_\delta\tt\mu_{n\dd}^{\dd,N}(|\cdot|^2)\delta +C_1^* \delta+\bar\Lambda^{i,\delta}({\bf X}_{n\delta}^{\delta}),
\end{align}
where the factor  $1-(  \kk/4 +\beta_\delta)\delta $ is positive  by taking $\delta\in(0,\delta^*_\kk]$ into consideration.

With \eqref{*-} at hand,
 we obtain  that for any integer $p\ge3,$
\begin{align*}
\big|X_{(n+1)\delta}^{\dd,i,N}\big|^{2p}
&\le  \big(\big(1-(  \kk/4 +\beta_\delta)\delta\big)\big|X_{n\delta}^{\dd,i,N}\big|^2+\beta_\delta\tt\mu_{n\dd}^{\dd,N}(|\cdot|^2)\delta\big)^p\\
&\quad+p\big(\big(1-(  \kk/4 +\beta_\delta)\delta\big)\big|X_{n\delta}^{\dd,i,N}\big|^2+\beta_\delta\tt\mu_{n\dd}^{\dd,N}(|\cdot|^2)\delta\big)^{p-1} \big(C_1^* \delta+\bar\Lambda^{i,\delta}({\bf X}_{n\delta}^{\delta})\big) \\
&\quad+\sum_{k=0}^{p-2}C_p^k\big(\big(1-(  \kk/4 +\beta_\delta)\delta\big)\big|X_{n\delta}^{\dd,i,N}\big|^2+\beta_\delta\tt\mu_{n\dd}^{\dd,N}(|\cdot|^2)\delta\big)^k \big(C_1^* \delta+\bar\Lambda^{i,\delta}({\bf X}_{n\delta}^{\delta})\big)^{p-k}\\
&=:\Upsilon_{n\delta}^{i}({\bf X}_{n\delta}^{\dd,N})+\hat\Upsilon_{n\delta}^{i}({\bf X}_{n\delta}^{\dd,N})+\bar\Upsilon_{n\delta}^{i}({\bf X}_{n\delta}^{\dd,N}).
\end{align*}
In the sequel, we aim to estimate separately the conditional expectations of $\Upsilon_{n\delta}^{i}$, $\hat\Upsilon_{n\delta}^{i}$, as well as $\bar\Upsilon_{n\delta}^{i}$ given the $\sigma$-algebra $\mathscr F_0^N$, which is generated by $X_0^{\delta,1,N}$, $\cdots,$ $X_0^{\delta,N,N}$.

In the first place,  the   binomial theorem and the Young inequality yield that
\begin{equation}\label{T3}
\begin{split}
\E\big(\Upsilon_{n\delta}^{i}({\bf X}_{n\delta}^{\dd,N})\big|\mathscr F_0^N\big)&=\sum_{k=0}^pC_p^k\big(1-(  \kk/4 +\beta_\delta)\delta\big)^k (\beta_\delta\delta)^{p-k} \big|X_{n\delta}^{\dd,i,N}\big|^{2k}  \big(\tt\mu_{n\dd}^{\dd,N}(|\cdot|^{2}) \big)^{(p-k)}\\
&\le \sum_{k=0}^pC_p^k\big(1-(  \kk/4 +\beta_\delta)\delta\big)^k (\beta_\delta\delta)^{p-k}\\
&\quad\times \Big(\frac{ k}{ p} \E\big(|X_{n\delta}^{\dd,i,N}|^{2p}\big|\mathscr F_0^N\big)  +\frac{ p- k}{ p} \E\big(\tt\mu_{n\dd}^{\dd,N}(|\cdot|^{2p})\big|\mathscr F_0^N\big) \Big)\\
&=(1-\kk\delta/4)^p\E\big(|X_{n\delta}^{\dd,i,N}|^{2p}\big|\mathscr F_0^N\big)\\
&\le(1-\kk\delta/4) \E\big(|X_{n\delta}^{\dd,i,N}|^{2p}\big|\mathscr F_0^N\big),
\end{split}
\end{equation}
where in the second identity we used the fact that $X_{ n\dd}^{\dd,i,N}$ and $X_{ n\dd}^{\dd,j,N}$ are identically distributed given $\mathscr F_0^N,$ and the last display is evident thanks to $1-\kk\delta/4\in(0,1).$
In the next place, due to $\E(\triangle W_{n\delta}^i|\mathscr F_0^N)=0$ and $\E(|\triangle W_{n\delta}^i|^2|\mathscr F_0^N)=d\delta$, it follows from  Young's inequality that there is a constant $C_2^*>0$  such that
\begin{equation}\label{T4}
\begin{split}
\E\big(\hat\Upsilon_{n\delta}^{i}({\bf X}_{n\delta}^{\dd,N})\big|\mathscr F_0^N\big)&\le pC_2^* \delta\E\big(\big(\big(1-(  \kk/4 +\beta_\delta)\delta\big)\big|X_{n\delta}^{\dd,i,N}\big|^2+\beta_\delta\tt\mu_{n\dd}^{\dd,N}(|\cdot|^2)\delta\big)^{p-1}\big|\mathscr F_0^N\big)\\
&\le\frac{\kk\dd }{16} \E\big(|X_{n\delta}^{\dd,i,N}|^{2p}\big|\mathscr F_0^N\big)+C_2^*\dd.
\end{split}
\end{equation}
Furthermore, note from \eqref{E*3}  and \eqref{B2} that
\begin{align*}
\big|b_1^\delta(x_i)+\mu_{{\bf x}}^N(b_0(x_i-\cdot))\big|\delta\le\big( \hat\lambda_{b_1}\delta^{ -{\ff12}}|x_i|+\hat C_{b_1}+ K(|x_i|+\mu_{{\bf x}}^N(|\cdot|))+|b_0({\bf0})|\big)\delta.
\end{align*}
This, together with the fact that the conditional expectation (given $\mathscr F_0^N$) of the  increment $\triangle W_{n\delta}^i$
contributes at least the order $\delta$, and the Young inequality,  leads to
\begin{align}\label{T5}
	\E\big(\bar\Upsilon_{n\delta}^{i}({\bf X}_{n\delta}^{\dd,N})\big|\mathscr F_0^N\big)&\le \frac{\kk\dd}{16} \E\big(|X_{n\delta}^{\dd,i,N}|^{2p}\big|\mathscr F_0^N\big)+C_3^*\dd
\end{align}
for some constant $C_3^*>0$. Ultimately, \eqref{T2} is reachable by pulling together \eqref{T3}, \eqref{T4} and \eqref{T5}.
\end{proof}

Based on Lemma \ref{lemma4}, it's turn to carry out the
\begin{proof}[Proof of Theorem \ref{thm5}]
Applying Theorem \ref{thm} with $r_0=0,$ $\theta_t=\bar\theta_t=t_\delta$, and $\tilde b(x,\mu)=b_1^\delta(x)+(b_0\ast\mu)(x)$, respectively,
enables us to derive that for  some $\lambda^*>0 $ and any $t\ge0,$
\begin{equation}\label{F6-1}
\begin{split}
\mathbb W_1\big(\mathscr L_{X_t^i},\mathscr L_{X_{t}^{\dd,i,N}}\big)&\lesssim  \e^{-\lambda^* t}\mathbb W_1\big(\mathscr L_{X_0^i},\mathscr L_{X_0^{\dd,i,N}}\big) +   N^{-\frac{1}{2}}\mathds{1}_{\{K>0\}}  \\
&\qquad\quad+ \int_0^t\e^{-\lambda^* (t-s)}\E\big| b(X_s^{\delta,i, N},\tilde \mu_s^{\delta,N})-\tilde b(X_{s_\delta}^{\delta,i,N},\tilde\mu_{s_\delta}^{\delta,N})\big| \,\d s.
\end{split}
\end{equation}
Next, by using \eqref{B2},  and \eqref{EE1}, it is easy to find that
 for $x,y\in\R^d$ and $\mu,\nu\in\mathscr P_1(\R^d)$,
\begin{align*}
|b(x,\mu)-\tilde b(y,\nu)|&\le |b_1(x)-b_1(y)|+|b_1(y)-b_1^\delta(y)|+|(b_0\ast\mu)(x)-(b_0\ast\nu)(y)|\\
&\le \int_0^1\<\nn b_1(y+s(x-y)),x-y\>\,\d s +\frac{\delta^{\frac{1}{2}}|b_1(y)|\cdot\|\nn b_1(y)\|_{\rm HS}}{1+\delta^{\frac{1}{2}}\|\nn b_1(y)\|_{\rm HS}} \\
&\quad+K\big(\mathbb W_1(\mu,\nu)+|x-y|\big)\\
&\lesssim  \big(1+|x|^{l^*}+|y|^{l^*}\big)|x-y|+\mathbb W_1(\mu,\nu)+(1+|y|^{2l^*+1})\delta^{\frac{1}{2}}.
\end{align*}
This, together with \eqref{F6-1}, implies that
\begin{equation}\label{F7}
\begin{split}
&\mathbb W_1\big(\mathscr L_{X_t^i},\mathscr L_{X_{t}^{\dd,i,N}}\big)\\&\lesssim  \e^{-\lambda^* t}\mathbb W_1\big(\mathscr L_{X_0^i},\mathscr L_{X_0^{\dd,i,N}}\big) +   N^{-\frac{1}{2}} \mathds{1}_{\{K>0\}} \\
&\quad+ \int_0^t\e^{-\lambda^* (t-s)}\E\Big(\E\big((1+|X_s^{\delta,i, N}|^{l^*}+|X_{s_\delta}^{\delta,i,N}|^{l^*})| X_s^{\delta,i, N}-X_{s_\delta}^{\delta,i,N}| \big|\mathscr F_0^N\big)\Big)\,\d s\\
&\quad+ \int_0^t\e^{-\lambda^* (t-s)}\E\big| X_s^{\delta,i, N}-X_{s_\delta}^{\delta,i,N}\big|\,\d s +\frac{1}{N}\sum_{j=1}^N\int_0^t\e^{-\lambda^* (t-s)}\E\big| X_s^{\delta,j, N}-X_{s_\delta}^{\delta,j,N}\big|\,\d s\\
&\quad+\delta^{\frac{1}{2}}\int_0^t\e^{-\lambda^* (t-s)}\big(1+\E|X_{s_\delta}^{\delta,i,N}|^{2l^*+1}\big)\,\d s.
\end{split}
\end{equation}
where  $\sigma$-algebra $\mathscr F_0^N$ is generated by $X_0^{\delta,1,N}$, $\cdots,$ $X_0^{\delta,N,N}$. Moreover, due to $\E|W_t^i-W_{t_\delta}^2|=d(t-t_\delta)$ and
$|b_1^\delta(x)|\lesssim1+\delta^{-\frac{1}{2}}|x|, x\in\R^d,$ by invoking \eqref{E*3},
it holds from Lemma \ref{lemma4}  that
\begin{align*}
\E\big(| X_t^{\delta,i, N}-X_{t_\delta}^{\delta,i,N}|^2\big|\mathscr F_0^N\big) \lesssim\big(1+\E\big(|X_{t_\delta}^{\delta,i,N}|^2\big|\mathscr F_0^N\big)\big)\delta\lesssim\big(1+|X_0^{\delta,i,N}|^2\big)\delta
\end{align*}
and
\begin{equation*}
\E\big(| X_t^{\delta,i, N}\big|^{2l^*}|\mathscr F_0^N\big) \lesssim 1+\E\big(| X_{t_\delta}^{\delta,i, N}|^{2l^*}|\mathscr F_0^N\big) \lesssim 1+| X_{0}^{\delta,i, N}|^{2l^*}.
\end{equation*}
 As a consequence,  the assertion \eqref{F3} is verifiable so the proof of Theorem \ref{thm5} is finished from \eqref{F7} and H\"older's inequality.
 \end{proof}

 \subsection{Proof of Theorem \ref{thm6}}

 The following lemma addresses the issue that the time grid associated with the adaptive EM scheme tends to infinity almost surely.

 \begin{lemma}\label{le}
Under  Assumptions $({\bf A}_1)$ and $({\bf A}_3)$,
\begin{align}\label{ER}
\P\big\{\oo\in \OO:\lim_{n\to \8}t_n(\oo)=+\8\big\}=1	.
\end{align}
\end{lemma}
\begin{proof}
Let for all $x\in\R^d$ and $\mu\in\mathscr P(\R^d)$,
\begin{align*}
h(x,\mu)=(1+|b(x,\mu)|^2)^{-1}.
\end{align*}
Then, it is easy to see that for   $x\in\R^d$ and $\mu\in\mathscr P(\R^d)$,
\begin{align*}
|b(x,\mu)|(1+|b(x,\mu)|)h(x,\mu)\le3/2.
\end{align*}
Next, from \eqref{B2} and (${\bf A_3}$), there is a constant $C_1>0$ such that for all $x\in\R^d$ and $\mu\in\mathscr P_1(\R^d)$,
\begin{align*}
 h(x,\mu)  \big(1+|x|^{2l^*+1}+\mu(|\cdot|)^2 \big)\le C.
\end{align*}
Furthermore, note that Assumption ({\bf T2}) in \cite[Proposition 4.1]{NTDH} can be weakened as below:
\begin{align*}
\<x,b(x,\mu)\>\le C_1\big(1+|x|^2+\mu(|\cdot|)^2\big),\quad x\in\R^d,~\mu\in\mathscr P_1(\R^d).
\end{align*}
Consequently, the assertion \eqref{ER} follows from \cite[Proposition 4.1]{NTDH} and \eqref{T10}.
\end{proof}

 Unlike the backward/tamed EM scheme, the step size involved is a constant so an inductive argument can be used to
establish the uniform moment bound. Nevertheless, with regard to the adaptive EM scheme \eqref{EW5}, the underlying step size is an adaptive stochastic process.
Hence, the approach adopted in treating Lemmas \ref{lemma3} and \ref{lemma4} no longer works to handle the uniform moment bound of the adaptive EM scheme \eqref{EW5}, which is stated as a lemma   below.

\begin{lemma}\label{lemma5}
Assume $({\bf A}_1)$ with $\phi(r)=\lambda_0r$ for some $\lambda_0>0$ and $\lambda>  2 K $. Then, for any $p\ge1$ , $i\in \mathbb S_N$ and $\dd\in(0,1)$, there is a constant $C^*_p>0$ such that
\begin{equation}\label{T6}
 \E\big|X_t^{\dd,i,N}\big|^p\le C^*_p\big(1+\E\big|X_0^{\dd,i,N}\big|^p\big),\quad t\ge0.
\end{equation}
\end{lemma}

\begin{proof}
By following partly  the proof of Lemma \ref{lemma3}, for each integer $p\ge3,$ it is sufficient to show that there exist constants $\lambda_p,C_p^*>0$
such that
\begin{align}\label{b3}
\mathbb{E} \big(|X_t^{\delta, i, N} |^{2p} \big|\mathscr F_0^N \big)\le C_p^*\big(1 +\e^{-\lambda_p t}\big|X_0^{\delta, i, N} \big|^{2p}\big),\quad i\in\mathbb S_N,~t\ge0
\end{align}
for the sake of validity of \eqref{T6}.

Applying It\^o's formula, we obtain from  \eqref{EW5} that for $\lambda_p:=p(\lambda-  2 K)/6 $ and integer $p\ge3$,
\begin{equation}\label{b1}
\d\big(\e^{\lambda_p t} |X_t^{\delta, i, N}|^{2p}\big) = \e^{\lambda_p t}\big(\lambda_p |X_t^{\delta, i, N}|^{2p}+\Phi^{i,p}({\bf X}_t^{\delta,N})
+c_p|X_t^{\delta, i, N}|^{2(p-1)} \big)\d t+\mathrm{d} M_t^{(p)},
\end{equation}
where $(M_t^{(p)})_{t\ge0}$ is a   martingale, $c_p:=\sigma^2 p(d+2(p-1))$, and
\begin{equation*}
\Phi^{i,p}\big({\bf X}_t^{\delta,N}\big):=2 p |X_t^{\delta, i, N}|^{2(p-1)}\big\< X_t^{\delta, i, N}, b(X_{{\underline t}}^{\delta, i, N}, \tilde{\mu}_{\underline t}^{\delta, N})\big\>.
\end{equation*}
By taking the structure of the adaptive step size defined in \eqref{T7} into consideration, besides $\delta\in(0,1),$ it is easy to see that
\begin{equation}\label{T11}
\big|b(X_{{\underline t}}^{\delta, i, N},\tilde{\mu}_{\underline t}^{\delta, N}) \big| (t-{\underline t})\le \delta^{\frac{1}{2}}(t-{\underline t})^{\frac{1}{2}}\le1.
\end{equation}
Whence, by making use of  the strong Markov property of $(W_t^i)_{t\ge0}$ and the tower property of conditional expectations, in addition to $t-\underline t\le1$, there exist constants $C_p^{*,1},C_p^{*,2}>0$ such that
\begin{equation}\label{b2}
\begin{split}
\E\big(|X_t^{\delta, i, N}|^{2p}\big|\mathscr F_0^N\big)&\le\frac{3}{2}\E\big(|X_{{\underline t}}^{\delta, i, N}|^{2p}\big|\mathscr F_0^N\big)+C_p^{*,1}\big( 1 + \E\big(| W_t^i-W_{{\underline t}}^i |^{2p}\big|\mathscr F_0^N\big) \big)\\
&=\frac{3}{2}\E\big(|X_{{\underline t}}^{\delta, i, N}|^{2p}\big|\mathscr F_0^N\big)+C_p^{*,1}\big( 1+ \E\big(\E\big(| W_t^i-W_{{\underline t}}^i |^{2p}\big|\mathscr F_{{\underline t}}^N\big)\big|\mathscr F_0^N\big) \big)\\
&\le \frac{3}{2}\E\big(|X_{{\underline t}}^{\delta, i, N}|^{2p}\big|\mathscr F_0^N\big)+C_p^{*,2},
\end{split}
\end{equation}
where $\mathscr F_{{\underline t}}^N$ is the $\si$-algebra of ${\underline t}$-past with $(\mathscr F_t^N)_{t\ge0}$ being the $\si$-algebra generated by $(X_0^{\delta, 1, N},\cdots,X_0^{\delta, N, N})$ and $(W_t^{1},\cdots, W_t^N)_{t\ge0}$. Then, provided that we claim that there is a  constant   $ C_p^{*,3}>0$
satisfying
\begin{align}\label{b4}
 \E\big(\Phi^{i,p}({\bf X}_t^{\delta,N})\big|\mathscr F_0^N\big)\le-\frac{1}{2}p(\lambda-2K)\E\big(|X_{{\underline t}}^{\delta, i, N}|^{2p}\big|\mathscr F_0^N\big)+C_p^{*,3},
\end{align}
  combining \eqref{b1} with \eqref{b2} yields that for some positive constants  $C_p^{*,4},C_p^{*,5}$,
\begin{align*}
\e^{\lambda_p t}\E\big( |X_t^{\delta, i, N}|^{2p}\big|\mathscr F_0^N\big)&\le \big|X_0^{\delta, i, N}\big|^{2p}+\int_0^t\e^{\lambda_p s}\Big(2\lambda_p \E\big( |X_s^{\delta, i, N}|^{2p}\big|\mathscr F_0^N\big)\\
&\qquad\qquad\qquad\qquad\qquad-\frac{1}{2}p(\lambda-2K)\E\big(|X_{\underline{s}}^{\delta, i, N}|^{2p}\big|\mathscr F_0^N\big)+C_p^{*,4}\Big)\,\d s\\
&\le \big|X_0^{\delta, i, N}\big|^{2p}+C_p^{*,5}(\e^{\lambda_p t}-1).
\end{align*}
Therefore, \eqref{b3} follows directly.

Below, we attempt to verify \eqref{b4}.
By invoking  \eqref{EW5} once more, it is apparent  to see that
\begin{equation}\label{T8}
\begin{split}
|X_t^{\delta, i, N}|^{2(p-1)}&=\big|X_{{\underline t}}^{\delta, i, N}+b(X_{{\underline t}}^{\delta, i, N},\tilde{\mu}_{\underline t}^{\delta, N})(t-{\underline t})+\si(W_t^i-W_{{\underline t}}^i)\big|^{2(p-1)}\\
%&=\big(\big|X_{{\underline t}}^{\delta, i, N}\big|^2+\Psi^i({\bf X}_t^{\delta,N})\big)^{p-1}\\
&=\big|X_{{\underline t}}^{\delta, i, N}\big|^{2(p-1)}+\sum_{k=0}^{p-2}C_{p-1}^k\big|X_{{\underline t}}^{\delta, i, N}\big|^{2k}\big(\Psi^i({\bf X}_{\underline t}^{\delta,N})\big)^{p-1-k}
%&\le \big(\big|X_{{\underline t}}^{\delta, i, N}\big|^2+h_{n_t}^\delta\delta+2\big|X_{{\underline t}}^{\delta, i, N}\big|(h_{n_t}^\delta\delta)^{\frac{1}{2}}\\
%&\quad+\si^2\big|W_t^i-W_{{\underline t}}^i\big|^2+2\si\<X_{{\underline t}}^{\delta, i, N},W_t^i-W_{{\underline t}}^i\>\\
%&\quad+2\si\<b(X_{{\underline t}}^{\delta, i, N},\tilde{\mu}_{\underline t}^{\delta, N}),W_t^i-W_{{\underline t}}^i\>(t-{\underline t})\big)^{p-1}\\
\end{split}
\end{equation}
and that
\begin{equation}\label{T9}
\begin{split}
\big\< X_t^{\delta, i, N}, b(X_{{\underline t}}^{\delta, i, N}, \tilde{\mu}_{\underline t}^{\delta, N})\big\> &=\big\<X_{{\underline t}}^{\delta, i, N},b(X_{{\underline t}}^{\delta, i, N}, \tilde{\mu}_{\underline t}^{\delta, N})\big\>+\Upsilon^i({\bf X}_{\underline t}^{\delta,N}),\\
\end{split}
\end{equation}
where
\begin{align*}
\Psi^i({\bf X}_{\underline t}^{\delta,N}):&=\big|b(X_{{\underline t}}^{\delta, i, N},\tilde{\mu}_{\underline t}^{\delta, N})\big|^2(t-{\underline t})^2+\si^2\big|W_t^i-W_{{\underline t}}^i\big|^2\\
&\quad+2\big\<X_{{\underline t}}^{\delta, i, N},b(X_{{\underline t}}^{\delta, i, N},\tilde{\mu}_{\underline t}^{\delta, N})\big\>(t-{\underline t})+2\si\big\<X_{{\underline t}}^{\delta, i, N},W_t^i-W_{{\underline t}}^i\big\>\\
&\quad+2\si\big\<b(X_{{\underline t}}^{\delta, i, N},\tilde{\mu}_{\underline t}^{\delta, N}),W_t^i-W_{{\underline t}}^i\big\>(t-{\underline t}),\\
\Upsilon^i({\bf X}_{\underline t}^{\delta,N}):&=\big|b(X_{{\underline t}}^{\delta, i, N},\tilde{\mu}_{\underline t}^{\delta, N})\big|^2(t-{\underline t})+\si\big\<W_t^i-W_{{\underline t}}^i,b(X_{{\underline t}}^{\delta, i, N}, \tilde{\mu}_{\underline t}^{\delta, N})\big\>.
\end{align*}
Now, plugging \eqref{T8} and \eqref{T9} back into \eqref{b1} gives that
\begin{equation*}
\begin{aligned}
\Phi^{i,p}({\bf X}_{\underline t}^{\delta,N})&= 2 p \big|X_{{\underline t}}^{\delta, i, N}\big|^{2(p-1)}\big\<X_{{\underline t}}^{\delta, i, N},b(X_{{\underline t}}^{\delta, i, N}, \tilde{\mu}_{\underline t}^{\delta, N})\big\>  +2 p \big|X_{{\underline t}}^{\delta, i, N}\big|^{2(p-1)}\Upsilon^i({\bf X}_t^{\delta,N})\\
&\quad+2 p \big(\big\<X_{{\underline t}}^{\delta, i, N},b(X_{{\underline t}}^{\delta, i, N}, \tilde{\mu}_{\underline t}^{\delta, N})\big\>+\Upsilon^i({\bf X}_t^{\delta,N})\big) \sum_{k=0}^{p-2}C_{p-1}^k\big|X_{{\underline t}}^{\delta, i, N}\big|^{2k}\big(\Psi^i({\bf X}_t^{\delta,N})\big)^{p-1-k}\\
&=:\Theta^{i,p}({\bf X}_{\underline t}^{\delta,N})+\hat\Theta^{i,p}({\bf X}_{\underline t}^{\delta,N})+\bar\Theta^{i,p}({\bf X}_{\underline t}^{\delta,N}).
\end{aligned}
\end{equation*}
By utilizing \eqref{T10}, the leading term $\Theta^{i,p}({\bf X}_{\underline t}^{\delta,N})$ can be tackled as follows:
\begin{equation}\label{T13}
\begin{split}
\E\big(\Theta^{i,p}({\bf X}_{\underline t}^{\delta,N})\big|\mathscr F_0^N\big)&\le 2p\Big(-\frac{1}{4}(3\lambda- 4 K)\E\big(|X_{{\underline t}}^{\delta, i, N}|^{2p}\big|\mathscr F_0^N\big)\\
&\quad+\frac{1}{2}K\E\big(|X_{{\underline t}}^{\delta, i, N}|^{2(p-1)}\tilde{\mu}_{\underline t}^{\delta, N}(|\cdot|^2)\big|\mathscr F_0^N\big)+C_0\E\big(|X_{{\underline t}}^{\delta, i, N}|^{2(p-1)}\big|\mathscr F_0^N\big)\Big)\\
&\le -p(\lambda-  2 K)\E\big(|X_{{\underline t}}^{\delta, i, N}|^{2p}\big|\mathscr F_0^N\big)+C_p^{*,6}
\end{split}
\end{equation}
for some constant $C_p^{*,6}>0,$
where we also employed that $X_{{\underline t}}^{\delta, i, N}$ and $X_{{\underline t}}^{\delta, j, N}$ are distributed identically given $\mathscr F_0^N.$ Subsequently, in view of
\begin{equation*}
\E\big(|X_{{\underline t}}^{\delta, i, N}|^{2(p-1)}\big\<W_t^i-W_{{\underline t}}^i,b(X_{{\underline t}}^{\delta, i, N}, \tilde{\mu}_{\underline t}^{\delta, N})\big\>\big|\mathscr F_0^N\big)=0,
\end{equation*}
 the Young inequality implies that for some constant $C_p^{*,7}>0,$
\begin{align}\label{T14}
\E\big(\hat\Theta^{i,p}({\bf X}_{\underline t}^{\delta,N})\big|\mathscr F_0^N\big)&\le   \frac{1}{4}p(\lambda- 2K)\E\big(|X_{\underline t}^{\delta, i, N}|^{2p}\big|\mathscr F_0^N\big)+C_p^{*,7}>0.
\end{align}
Furthermore, by virtue of \eqref{T11}, we obviously obtain that
\begin{align*}
\Psi^i({\bf X}_{\underline t}^{\delta,N}) &\le  t-{\underline t} +\si^2\big|W_t^i-W_{{\underline t}}^i\big|^2+2 (t-{\underline t}\,)^{\frac{1}{2}}\big|X_{{\underline t}}^{\delta, i, N}\big|
+2|\si|\big(1+\big|X_{{\underline t}}^{\delta, i, N}\big|\big)  \big|W_t^i-W_{{\underline t}}^i\big|
\end{align*}
and
\begin{align*}
\Upsilon^i({\bf X}_{\underline t}^{\delta,N})\le   1+|\si|\cdot\big|b(X_{{\underline t}}^{\delta, i, N}, \tilde{\mu}_{\underline t}^{\delta, N})\big|\cdot\big|W_t^i-W_{{\underline t}}^i\big|.
\end{align*}
Thus, we find that
\begin{equation}\label{T12}
\begin{split}
\E\big(\bar\Theta^{i,p}({\bf X}_{\underline t}^{\delta,N})\big|\mathscr F_0^N\big)&\le 2 p \big(1+\big(|X_{{\underline t}}^{\delta, i, N}\big| +|\si||W_t^i-W_{{\underline t}}^i|\big)\big|b(X_{{\underline t}}^{\delta, i, N}, \tilde{\mu}_{\underline t}^{\delta, N})\big|\big)\\
&\quad\times\sum_{k=0}^{p-2}C_{p-1}^k\big|X_{{\underline t}}^{\delta, i, N}\big|^{2k}\Big( t-{\underline t} +\si^2|W_t^i-W_{{\underline t}}^i|^2+2\delta |X_{{\underline t}}^{\delta, i, N}| (t-{\underline t}\,)^{\frac{1}{2}}\\
&\qquad\qquad\qquad\qquad\qquad+2|\si|\big(1+|X_{{\underline t}}^{\delta, i, N}|\big)\big|W_t^i-W_{{\underline t}}^i\big|\Big)^{p-1-k}.
\end{split}
\end{equation}
Notice that the degree of the polynomial (on the right hand side of \eqref{T12}) with respect to $\big|X_{{\underline t}}^{\delta, i, N}\big| $ is $2(p-1)$. In addition, the conditional expectation of the polynomial with respect to  $ |W_t^i-W_{{\underline t}}^i | $ given the $\si$-algebra $\mathscr F_{\underline t}^N$ offers at least the order $(t-{\underline t}\,)^{\frac{1}{2}}$ so that the term $ |b (X_{{\underline t}}^{\delta, i, N}, \tilde{\mu}_{\underline t}^{\delta, N} ) |(t-{\underline t}\,)^{\frac{1}{2}}$ can be uniformly bounded by taking advantage of \eqref{T11}. Once again, with the aid of  Young's inequality, there exists a  constant  $C_p^{*,8}>0$ such that
\begin{align*}
\E\big(\bar\Theta^{i,p}({\bf X}_{\underline t}^{\delta,N})\big|\mathscr F_0^N\big)&\le   \frac{1}{4}p(\lambda- 2K)\E\big(|X_{{\underline t}}^{\delta, i, N}|^{2p}\big|\mathscr F_0^N\big)+C_p^{*,8}>0.
\end{align*}
This, together with \eqref{T13} and \eqref{T14}, enables us to achieve \eqref{b4}.
\end{proof}

Before the end of this work, we accomplish the
\begin{proof}[Proof of Theorem \ref{thm6}]
Once again, an application of Theorem \ref{thm} with $r_0=0,$ $\theta_t=\bar\theta_t=\underline t$, and $\tilde b=b$, respectively,
yields  that for  some $\lambda^*>0 $ and any $t\ge0,$
\begin{equation}\label{F6}
\begin{split}
\mathbb W_1\big(\mathscr L_{X_t^i},\mathscr L_{X_{t}^{\dd,i,N}}\big)&\lesssim  \e^{-\lambda^* t}\mathbb W_1\big(\mathscr L_{X_0^i},\mathscr L_{X_0^{\dd,i,N}}\big) +   N^{-\frac{1}{2}}\mathds{1}_{\{K>0\}}  \\
&\qquad\quad+ \int_0^t\e^{-\lambda^* (t-s)}\E\big| b(X_s^{\delta,i, N},\tilde \mu_s^{\delta,N})- b(X_{\underline s}^{\delta,i,N},\tilde\mu_{\underline s}^{\delta,N})\big| \,\d s.
\end{split}
\end{equation}
Next, from \eqref{B2} and (${\bf A}_3$), it is easy to see that
\begin{equation}\label{F10}
\begin{split}
\E\big| b(X_t^{\delta,i, N},\tilde \mu_t^{\delta,N})- b(X_{\underline t}^{\delta,i,N},\tilde\mu_{\underline t}^{\delta,N})\big|&\lesssim \E\big(\big(1+\big|X_t^{\delta,i, N}\big|^{l^*}+\big|X_{\underline t}^{\delta,i,N}\big|^{l^*}\big)\big|X_t^{\delta,i, N}-X_{\underline t}^{\delta,i,N}\big|\big)\\
&\quad+\E\big|X_t^{\delta,i, N}-X_{\underline t}^{\delta,i,N}\big|+\frac{1}{N}\sum_{j=1}^N\E\big|X_t^{\delta,j, N}-X_{\underline t}^{\delta,j,N}\big|.
\end{split}
\end{equation}
Moreover, by taking the definition of $h^\delta$ given in \eqref{T7} into account, one has
\begin{align}\label{F9}
|X_t^{\delta,i, N}-X_{\underline t}^{\delta,i,N}|&\le|b(X_{\underline t}^{\delta,i,N},\tilde\mu_{\underline t}^{\delta,N})|(t-\underline t)+|\si(W_t^i-W_{\underline t}^i)| \le \delta+|\si(W_t^i-W_{\underline t}^i)|.
\end{align}
 Note that from H\"older's inequality, we deduce that
 \begin{equation}\label{F11}
 \begin{split}
 	&\E\big((1+|X_t^{\delta,i, N}|^{l^*})\big|X_t^{\delta,i, N}-X_{\underline t}^{\delta,i,N}\big|\big)\\
 &\1 \E\Big(\big(1+\big(\E\big(|X_t^{\delta,i, N}|^{2l^*}\big|\mathscr F_0^N\big)\big)^{\ff12}\big)
	\big(\E\big(\big|X_t^{\delta,i, N}-X_{\underline t}^{\delta,i,N}\big|^2\big|\mathscr F_0^N\big)\big)^{\ff12}\Big).
\end{split}
 \end{equation}
Whereafter, the assertion \eqref{F8} can be attainable by plugging \eqref{F10} back into \eqref{F6} followed by combining \eqref{F9} and \eqref{F11} with  Lemma \ref{lemma5}. Consequently, the proof of Theorem \ref{thm6} is complete.
\end{proof}

\noindent{\bf Acknowledgments}

The research of
Bao is supported by the National Key R\&D Program of China (2022YFA1006004) and
NSF of China (No. 12071340).

\end{document}